\documentclass[amstex,12pt]{article}
\usepackage{}
\usepackage{amsmath}
\usepackage{amssymb,amsmath,amsthm}
\newtheorem{theorem}{Theorem}[section]

\newtheorem{lemma}{Lemma}[section]
\newtheorem{definition}{Definition}[section]
\newtheorem{remark}{Remark}[section]

\newtheorem{example}{Example}[section]

\newcommand{\beq}{\begin{equation}}
\newcommand{\eeq}{\end{equation}}
\newcommand{\beqn}{\begin{eqnarray}}
\newcommand{\eeqn}{\end{eqnarray}}

\allowdisplaybreaks

\begin{document}

\title{Permanence and almost periodic solution of a multispecies Lotka-Volterra mutualism system with time varying  delays on time scales\thanks{This work is supported by the National
Natural Sciences Foundation of People's Republic of China under
Grant 11361072.}}
\author {Yongkun Li\thanks{%
The corresponding author. Email: yklie@ynu.edu.cn.} and Pan Wang\\
Department of Mathematics, Yunnan University\\
Kunming, Yunnan 650091\\
 People's Republic of China}
\date{}
\maketitle{}
\begin{abstract}
In this paper, we consider the almost periodic dynamics of a multispecies Lotka-Volterra mutualism system with time varying  delays on time scales.
By establishing some dynamic inequalities on time scales, a permanence result for the model is obtained. Furthermore,
by means of the almost periodic functional hull theory on time scales and Lyapunov functional, some criteria are obtained for the existence, uniqueness and global attractivity of almost periodic solutions of the model. Our results complement and extend some scientific work in recent years. Finally, an example is given to illustrate the main results.
\end{abstract}

\section{Introduction}

\setcounter{equation}{0}
{\setlength\arraycolsep{2pt}}
 \indent

Recently, there are many scholars concerning with the dynamics of the mutualism model.
Topics such as permanence, global attractivity, and periodicity of mutualism systems governed by differential equations
were extensively investigated (see[1-10]). For example, in \cite{7}, the author studied the existence of positive periodic solutions of the periodic mutualism model:{\setlength\arraycolsep{2pt}\begin{eqnarray}\label{eli1}
\left\{
\begin{array}{lll}
{\displaystyle\frac{\mathrm{d}x_{1}(t)}{\mathrm{d}t}=x_{1}(t)\bigg[\frac{r_{1}(t)K_{1}(t)+r_{1}(t)\alpha_{1}(t)x_{2}(t-\tau_{2}(t))}
{1+x_{2}(t-\tau_{2}(t))}}-r_{1}(t)x_{1}(t-\sigma_{1}(t))\bigg],\\
{\displaystyle\frac{\mathrm{d}x_{2}(t)}{\mathrm{d}t}=x_{2}(t)\bigg[\frac{r_{2}(t)K_{2}(t)+r_{2}(t)\alpha_{2}(t)x_{1}(t-\tau_{1}(t))}
{1+x_{1}(t-\tau_{1}(t))}}-r_{2}(t)x_{2}(t-\sigma_{2}(t))\bigg],
\end{array}
\right.
\end{eqnarray}}where $r_{i}, K_{i}, \alpha_{i}\in C(\mathbb{R},\mathbb{R}^{+})$, $\alpha_{i} >K_{i}, i=1,2$, $\tau_{i}, \sigma_{i}\in C(\mathbb{R},\mathbb{R}^{+}),i=1,2$, $r_{i}, K_{i}, \alpha_{i}, \tau_{i}, \sigma_{i}(i=1,2)$ are functions of period $\omega>0$.

However, in applications, if the various
constituent components of the temporally nonuniform environment
is with incommensurable periods, then one has to consider the environment to be almost periodic since
there is no a priori reason to expect the existence
of periodic solutions. Hence, if we consider the
effects of the environmental factors, almost periodicity is sometimes
more realistic and more general than periodicity.
In recent years, the almost periodic solution of
the models in biological populations has been studied
extensively (see [11-18] and the references cited therein).
In addition, some recent attention was on the permanence and global stability of discrete mutualism system, and many excellent results have been derived (see [19-24]). In \cite{19}, the authors considered the following discrete multispecies Lotka-Volterra mutualism system:
{\setlength\arraycolsep{2pt}\begin{eqnarray}\label{e12}
x_{i}(k+1)=x_{i}(k)\exp\bigg\{a_{i}(k)-b_{i}(k)x_{i}(k)+\sum_{j=1,j\neq i}^{n}c_{ij}(k)\displaystyle\frac{x_{j}(k)}{d_{ij}+x_{j}(k)}\bigg\},\,\,i=1, 2,\ldots,
\end{eqnarray}}
where $x_{i}(k)$ stand for the densities of species $x_{i}$ at the $k$th generation, $a_{i}(k)$ represent the natural growth rates of species $x_{i}$ at the $k$th generation, $b_{i}(k)$ are the intraspecific effects of the $k$th generation of species $x_{i}$ on own population, $c_{ij}(k)$ measure the interspecific mutualism effects of the $k$th generation of species $x_{j}$ on species $x_{i}$ $(i, j=1,2,\ldots,n,i\neq j)$, and $d_{ij}(\geq 1)$ are positive control constants. By means of the theory of difference inequality and Lyapunov function, sufficient conditions are established for the existence and uniformly asymptotic stability of unique positive almost periodic solution to system $\eqref{e12}$.

Furthermore, so many processes, both natural and manmade, in biology, medicine, chemistry, physics, engineering, economics, etc. involve time delays. Time delays occur so often so if we ignore them, we ignore reality. Generally, the meaning of time delay is that some time elapses between causes and their effects (for instance, in population dynamics, individuals always need some time to mature, or in medicine, infectious diseases have incubation periods). Specially, in
the real world, the delays in differential equations of biological phenomena are usually time-varying. Thus, it is worthwhile continuing to study the existence and stability of a unique almost periodic solution of the multispecies Lotka-Volterra mutualism system with time varying delays.

Since permanence is one of the most important topics on the study of population dynamics, one of the most interesting questions in mathematical biology concerns the survival of species in ecological models. Biologically, when a system of interacting species is persistent in a suitable sense, it means that all the species survive in the long term. It is reasonable to ask for conditions under which the system is permanent.

Also, as we known, the study of dynamical systems on time scales is now an active area of research. The theory of times scales has received a lot of attention which was introduced by Stefan Hilger in his Ph.D. thesis in 1988, providing a rich theory that unifies and extends continuous and discrete analysis \cite{20}. In fact, both continuous and discrete systems are very important in implementation and applications. But it is troublesome to study the dynamics for continuous and discrete systems respectively. Therefore, it is significate to study that on time scales which can unify the continuous and discrete situations.

Motivated by the above reasons, in this paper, we are concerned with the following multispecies Lotka-Volterra mutualism system with time varying delays on time scales:
{\setlength\arraycolsep{2pt}\begin{eqnarray}\label{e13}
x_{i}^{\Delta}(t)=a_{i}(t)-b_{i}(t)e^{x_{i}(t-\tau_{i}(t))}+\sum_{j=1,j\neq i}^{n}c_{ij}(t)\frac{e^{x_{j}(t-\delta_{j}(t))}}{d_{ij}+e^{x_{j}(t-\delta_{j}(t))}},\,\,i=1,2,\ldots, n,\,\,t\geq t_{0},\,\,t, t_{0}\in{\mathbb{T}},\nonumber\\
\end{eqnarray}}where
$\mathbb{T}$ is an almost periodic time scale.

\begin{remark}\label{r1}
Let $y_{i}(t)=e^{x_{i}(t)}$, if $\mathbb{T}=\mathbb{R}$, then system (1.3) is reduced to the following system:
{\arraycolsep=2pt
\begin{equation}\label{ea1}
 \addtolength{\arraycolsep}{-3pt}
y_{i}(t)=y_{i}(t)\bigg\{a_{i}(t)-b_{i}(t)y_{i}(t-\tau_{i}(t))+\sum_{j=1,j\neq i}^{n}c_{ij}(t)\frac{y_{j}(t-\delta_{j}(t))}{d_{ij}+y_{j}(t-\delta_{j}(t))}\bigg\},\,\,i=1,2,\ldots, n,\,\,t\in\mathbb{R},
\end{equation}}
which is a generalization of \eqref{eli1}.
If $\mathbb{T}=\mathbb{Z}$,
then system (1.3) is reduced to the following system:
{\arraycolsep=2pt
\begin{eqnarray}\label{ea2}
y_{i}(k+1)&=&y_{i}(k)\exp\bigg\{a_{i}(k)-b_{i}(k)y_{i}(k-\tau_{i}(k))\nonumber\\
&&+\sum_{j=1,j\neq i}^{n}c_{ij}(k)\frac{y_{j}(k-\delta_{j}(k))}{d_{ij}+y_{j}(k-\delta_{j}(k))}\bigg\},\,\,i=1, 2,\ldots,\,\,t\in\mathbb{Z},
\end{eqnarray}}
let $\tau_{i}(k)=0$, $\delta_{j}(k)=0$, then system (1.3) is reduced to  system (1.2).
\end{remark}

By the biological meaning, we will focus our discussion on the positive solutions of system $\eqref{e13}$. So, it is assumed that the initial
condition of system $\eqref{e13}$ is the form
\begin{eqnarray}\label{e14}
x_{i}(s)=\varphi_{i}(s)\geq0, \quad\varphi_{i}(t_{0})>0,\quad s\in[t_{0}-\theta, t_{0}]_{\mathbb{T}},\quad i=1,2,\ldots,n,
\end{eqnarray}where $\theta=\max\{\tau^{+}, \delta^{+}\}$, $\tau^{+}=\max\limits_{1\leq i\leq n}\sup\limits_{t\in\mathbb{T}}\{\tau_{i}(t)\}$, $\tau^{-}=\min\limits_{1\leq i\leq n}\inf\limits_{t\in\mathbb{T}}\{\tau_{i}(t)\}$, $\delta^{+}=\max\limits_{1\leq j\leq n}\sup\limits_{t\in\mathbb{T}}\{\delta_{j}(t)\}$, $\delta^{-}=\min\limits_{1\leq j\leq n}\inf\limits_{t\in\mathbb{T}}\{\delta_{j}(t)\}$.

For convenience, we denote
\begin{eqnarray*}
 f^l=\inf_{t\in{\mathbb{T}}}|f(t)|,\quad f^u=\sup_{t\in{\mathbb{T}}}|f(t)|.
\end{eqnarray*}

Throughout this paper, we assume that
\begin{enumerate}
  \item [$(H_{1})$]
  $a_{i}(t), b_{i}(t), c_{ij}(t)$, $\tau_{i}(t)$, $\delta_{j}(t)$ are all  almost periodic functions such that $a_{i}^{l}>0$, $b_{i}^{l}>0$, $c_{ij}^{l}>0$, $\tau^{-}>0$ and $\delta^{-}>0$; $d_{ij}>1$,  $t-\tau_{i}(t)\in\mathbb{T}$ and $t-\delta_{j}(t)\in\mathbb{T}$ for $t\in{\mathbb{T}}$, $i,j=1,2,\ldots,n$, $j\neq i$.
  \item [$(H_{2})$]
    $\tau^{\Delta}=\max\limits_{1\leq i\leq n}\sup\limits_{t\in\mathbb{T}}\{\tau_{i}^{\Delta}(t)\}$, $\delta^{\Delta}=\max\limits_{1\leq j\leq n}\sup\limits_{t\in\mathbb{T}}\{\delta_{j}^{\Delta}(t)\}$ and $1-\tau^{\Delta}>0$, $1-\delta^{\Delta}>0$.
\end{enumerate}

To the best of our knowledge, there is no paper published on  the permanence, the  existence and uniqueness of globally attractive   almost periodic solutions to   systems \eqref{ea1} and \eqref{ea2}.
The main purpose of this paper is by establishing some dynamic inequalities on time scales to discuss the permanence of system \eqref{e13} and by using the almost periodic functional hull theory on time scales to establish criteria for the existence and uniqueness of globally attractive   almost periodic solutions of   system \eqref{e13}.

The paper is organized as follows. In Section 2, we introduce some basic definitions, necessary lemmas and establishing some dynamic inequalities on time scales which will
be used in later sections. In Section 3, we  discuss the permanence of system $\eqref{e13}$. In Section 4, we consider the global attractivity of almost periodic solutions of  system $\eqref{e13}$ by means of Lyapunov functional. In Section 5, some sufficient conditions are obtained for the existence of positive almost periodic solutions of system $\eqref{e13}$ by use of the almost periodic functional hull theory on time scales. The main result in Sections 4 and  5 are illustrated by giving an example in Section 6.

\section{Preliminaries}

\setcounter{equation}{0}
{\setlength\arraycolsep{2pt}}
 \indent

In this section, we shall recall some basic definitions, lemmas which are used in what follows.

A time scale $\mathbb{T}$ is an arbitrary nonempty closed subset of the real numbers, the forward and backward jump operators $\sigma$, $\rho:\mathbb{T}\rightarrow \mathbb{T}$ and the forward graininess $\mu:\mathbb{T}\rightarrow \mathbb{R}^{+}$ are defined, respectively, by
\[
\sigma(t):=\inf \{s\in\mathbb{T}:s> t\},\,\,\rho(t):=\sup\{s\in\mathbb{T}:s<t\}\,\,
\text{and}\,\,\mu(t)=\sigma(t)-t.
\]

A point $t\in\mathbb{T}$ is called left-dense if $t>\inf\mathbb{T}$
and $\rho(t)=t$, left-scattered if $\rho(t)<t$, right-dense if
$t<\sup\mathbb{T}$ and $\sigma(t)=t$, and right-scattered if
$\sigma(t)>t$. If $\mathbb{T}$ has a left-scattered maximum $m$,
then $\mathbb{T}^k=\mathbb{T}\setminus\{m\}$; otherwise
$\mathbb{T}^k=\mathbb{T}$. If $\mathbb{T}$ has a right-scattered
minimum $m$, then $\mathbb{T}_k=\mathbb{T}\setminus\{m\}$; otherwise
$\mathbb{T}_k=\mathbb{T}$.

A function $f:\mathbb{T}\rightarrow\mathbb{R}$ is right-dense
continuous provided it is continuous at right-dense point in
$\mathbb{T}$ and its left-side limits exist at left-dense points in
$\mathbb{T}$. If $f$ is continuous at each right-dense point and
each left-dense point, then $f$ is said to be continuous function on
$\mathbb{T}$.

For $y:\mathbb{T}\rightarrow\mathbb{R}$ and $t\in\mathbb{T}^k$, we
define the delta derivative of $y(t)$, $y^\Delta(t)$, to be the
number (if it exists) with the property that for a given
$\varepsilon>0$, there exists a neighborhood $U$ of $t$ such that
\[
|[y(\sigma(t))-y(s)]-y^\Delta(t)[\sigma(t)-s]|<\varepsilon|\sigma(t)-s|
\]
for all $s\in U$.

If $y$ is continuous, then $y$ is right-dense continuous, and if $y$
is delta differentiable at $t$, then $y$ is continuous at $t$.

Let $f$ be right-dense continuous, if $F^{\Delta}(t)=f(t)$, then we define the delta integral by
\begin{eqnarray*}
\int_{r}^{s}f(t)\Delta t=F(s)-F(r),\,\,\,\,\,\,r, s\in{\mathbb{T}}.
\end{eqnarray*}
\begin{lemma}\label{lem21}\cite{20}
Assume $f, g : \mathbb{T} \longrightarrow \mathbb{R}$ are delta differentiable at $t\in\mathbb{T}_{¦Ê}$, then
\begin{itemize}
    \item  [$(i)$]  $(f + g)^{\Delta}(t)=f^{\Delta}(t) + g^{\Delta}(t)$;
    \item  [$(ii)$] $(fg)^{\Delta}(t)=f^{\Delta}(t)g(t) + f^{\sigma}(t)g^{\Delta}(t) = f(t)g^{\Delta}(t) + f^{\Delta}(t)g^{\sigma}(t)$;
    \item  [$(iii)$] If $g(t)g^{\sigma}(t)\neq 0$, then ${\displaystyle\bigg(\frac{f}{g}\bigg)^{\Delta}= \frac{f^{\Delta}(t)g(t)-f(t)g^{\Delta}(t)}{g(t)g^{\sigma}(t)}}$;
    \item  [$(iv)$] If $f$ and $f^{\Delta}$ are continuous, then $(\int_{a}^{t}f(t, s)\Delta s)^{\Delta}=f(\sigma(t), t)+\int_{a}^{t}f^{\Delta}(t, s)\Delta s$.
\end{itemize}
\end{lemma}

A function $p:\mathbb{T}\rightarrow\mathbb{R}$ is called regressive provided $1+\mu(t)p(t)\neq 0$ for all $t\in{\mathbb{T}^{k}}$. The
set of all regressive and rd-continuous functions $p:\mathbb{T}\rightarrow\mathbb{R}$ will be denoted by $\mathcal{R}=\mathcal{R}(\mathbb{T})=\mathcal{R}(\mathbb{T}, \mathbb{R})$. We define the set $\mathcal{R}^{+}=\mathcal{R}^{+}(\mathbb{T}, \mathbb{R})=\{p\in\mathcal{R}: 1+\mu(t)p(t)> 0, \forall t\in{\mathbb{T}}\}$.

If $r\in\mathcal{R}$, then the generalized exponential function $e_{r}$ is defined by
\begin{eqnarray*}
e_{r}(t, s)=\exp\bigg\{\int_s^t\xi_{\mu(\tau)}(r(\tau))\Delta\tau\bigg\},
\end{eqnarray*}
for all $s,t\in\mathbb{T}$, with the cylinder transformation
\begin{eqnarray*}
\xi_h(z)=\bigg\{\begin{array}{ll} {\displaystyle\frac{\mathrm{Log}(1+hz)}{h}},\,\,h\neq 0,\\
z,\,\,\,\,\,\,\,\quad\quad\quad\quad\quad h=0.\\
\end{array}
\end{eqnarray*}

Let $p,q:\mathbb{T}\rightarrow\mathbb{R}$ be two regressive functions, we define
\[
p\oplus q=p+q+\mu pq,\,\,\,\,\ominus p=-\frac{p}{1+\mu p},\,\,\,\,p\ominus q=p\oplus(\ominus q)=\frac{p-q}{1+\mu q}.
\]
Then the generalized exponential function has the following
properties.

\begin{lemma}\label{lem22}\cite{20}
Assume that $p,q:\mathbb{T}\rightarrow\mathbb{R}$ are two regressive
functions, then
\begin{itemize}
    \item  [$(i)$]  $e_{0}(t,s)\equiv 1$ and $e_p(t,t)\equiv 1$;
    \item  [$(ii)$] $e_p(\sigma(t),s)=(1+\mu(t)p(t))e_p(t,s)$;
    \item  [$(iii)$]$e_p(t,s)=1/e_p(s,t)=e_{\ominus p}(s,t)$;
    \item  [$(iv)$]  $e_p(t,s)e_p(s,r)=e_p(t,r)$;
    \item  [$(v)$] $e_p(t,s)e_q(t,s)=e_{p\oplus q}(t,s)$;
    \item  [$(vi)$] $e_p(t,s)/e_q(t,s)=e_{p\ominus q}(t,s)$;
    \item  [$(vi)$] $\big(\frac{1}{e_p(t,s)}\big)^{\Delta}=\frac{-p(t)}{e^{\sigma}_p(t,s)}$.
\end{itemize}
\end{lemma}

\begin{lemma}\label{lem24}\cite{zhang1} Let $f:\mathbb{T}\rightarrow \mathbb{R}$ be a continuously increasing function and $f(t)>0$ for $t\in\mathbb{T}$, then
\begin{eqnarray*}
\frac{f^{\Delta}(t)}{f^{\sigma}(t)}\leq[\ln (f(t))]^{\Delta}\leq\frac{f^{\Delta}(t)}{f(t)}.
\end{eqnarray*}
\end{lemma}

\begin{definition}\label{def21}\cite{24}  A time scale $\mathbb{T}$ is called an almost periodic time scale if
\begin{eqnarray*}
\Pi=\big\{\tau\in\mathbb{R}: t\pm\tau\in\mathbb{T}, \forall t\in{\mathbb{T}}\big\}\neq\{0\}.
\end{eqnarray*}
\end{definition}

Throughout this paper, $\mathbb{E}^{n}$ denotes $\mathbb{R}^{n}$ or $\mathbb{C}^{n}$, $D$ denotes an open set in $\mathbb{E}^{n}$ or $D=\mathbb{E}^{n}$,
and $S$ denotes an arbitrary compact subset of $D$.

\begin{definition}\label{def22}\cite{24}   Let $\mathbb{T}$ be an almost periodic time scale. A function $f\in C(\mathbb{T}\times D, \mathbb{E}^{n})$ is called an
almost periodic function in $t\in\mathbb{T}$ uniformly for $x\in D$ if the $\epsilon$-translation set of $f$
\begin{eqnarray*}
E\{\epsilon, f, S\}=\{t\in\Pi: |f(t+\tau, x)-f(t, x)|<\epsilon, \forall (t, x)\in\mathbb{T}\times S\}
\end{eqnarray*}
is a relatively dense set in $\mathbb{T}$ for all $\epsilon>0$ and for each compact subset $S$ of $D$; that is, for any given $\epsilon>0$ and for each compact subset $S$ of $D$, there exists a constant
$l(\epsilon, S)>0$ such that each interval of length $l(\epsilon, S)$ contains a $\tau(\epsilon, S)\in E\{\epsilon, f, S\}$ such that
\begin{eqnarray*}
 |f(t+\tau, x)-f(t, x)|<\epsilon, \forall(t, x)\in\mathbb{T}\times S.
\end{eqnarray*}
$\tau$ is called the $\epsilon$-translation number of $f$ and $l(\epsilon, S)$ is called the inclusion length of $E\{\epsilon, f, S\}$.
\end{definition}

For convenience, we denote $AP(\mathbb{T})=\{f: f\in C(\mathbb{T}, \mathbb{E}^{n})$, $f$ is almost periodic$\}$ and introduce some notations: let $\alpha=\{\alpha_{n}\}$ and $\beta=\{\beta_{n}\}$ be two sequences. Then, $\beta\subset\alpha$ means that $\beta$ is a subsequence of $\alpha$, $\alpha+\beta=\{\alpha_{n}+\beta_{n}\}$, $-\alpha=\{-\alpha_{n}\}$. $\alpha$ and $\beta$ are common subsequences of $\alpha'$ and $\beta'$, respectively, which means that $\alpha_{n}=\alpha'_{n(k)}$ and $\beta_{n}=\beta'_{n(k)}$ for some given function $n(k)$.

We will introduce the translation operator $T$, $T_{\alpha}f(t, x)=g(t, x)$, which means that $g(t, x)=\lim_{n\rightarrow+\infty}f(t+\alpha_{n}, x)$ and is written only when the limit exists. The mode of convergence, for example, pointwise, uniform, and so forth, will be specified at each use of the symbol.

\begin{definition}\label{def23}\cite{24}
Let $f\in C(\mathbb{T}\times D, \mathbb{E}^{n})$, $H(f)=\{g: \mathbb{T}\times D\rightarrow\mathbb{E}^{n}|$ there exits $\alpha\in\Pi$ such that
$T_{\alpha}f(t, x)=g(t, x)$ exists uniformly on $\mathbb{T}\times S\}$ is called the hull of $f$.
\end{definition}

\begin{lemma}\label{lem26}\cite{24}
If $f(t, x)$ is almost periodic  in $t\in\mathbb{T}$ uniformly for $x\in D$, then, for any $g(t, x)\in H(f)$, $g(t, x)$ is almost periodic  in $t\in\mathbb{T}$ uniformly for $x\in D$.
\end{lemma}

\begin{lemma}\label{lem27}\cite{24}
If $f(t, x)$ is almost periodic  in $t\in\mathbb{T}$ uniformly for $x\in D$, denote $F(t, x)=\int_{0}^{t}f(s, x)\Delta s$, then $F(t, x)$ is almost periodic in $t\in\mathbb{T}$ uniformly for $x\in D$ if and only if $F(t, x)$ is bounded on $\mathbb{T}\times S$.	
\end{lemma}

\begin{lemma}\label{lem28}\cite{24}
A function $f(t, x)$ is almost periodic in $t\in\mathbb{T}$ uniformly for $x\in D$  if and only if from every pair of sequences $\alpha^{'}\subset\Pi$, $\beta^{'}\subset\Pi$ one can extract common subsequences $\alpha\subset\alpha^{'}$, $\beta\subset\beta^{'}$ such that
 \begin{eqnarray*}
 T_{\alpha+\beta}f(t, x)=T_{\alpha}T_{\beta}f(t, x).
  \end{eqnarray*}
\end{lemma}

\begin{lemma}\label{lem29}\cite{24}
A function $f(t)$ is almost periodic   if and only if for any sequence $\{\alpha_{n}^{'}\}\subset\Pi$ there exists a subsequence $\{\alpha_{n}\}\subset\{\alpha_{n}^{'}\}$ such that $f(t+\alpha_{n})$ converges uniformly on $t\in\mathbb{T}$ as
$n\rightarrow\infty$. Furthermore, the limit function is also   almost periodic.
\end{lemma}

Consider the following equation
\begin{eqnarray}\label{e21}
x^{\Delta}(t)=f(t,x),\,\,t\in\mathbb{T}
\end{eqnarray}
and the corresponding hull equation
\begin{eqnarray}\label{e22}
x^{\Delta}(t)=g(t,x),\,\,t\in\mathbb{T},
\end{eqnarray}
where $f:\mathbb{T} \times S\rightarrow \mathbb{E}^{n}$, $f(t, x)$ is almost periodic in $t$ uniformly for $x\in S$, $g(t,x)\in H(f)$.

Similar to the proof of Theorem 3.2 in \cite{25}, one can easily get the following.
\begin{lemma}\label{lem210}    Let
$f(t, x)\in C(\mathbb{T}\times S, \mathbb{E}^{n})$ be an almost periodic in $t$ uniformly for $x\in S$. For every $g(t,x)\in H(f)$, the hull equation $\eqref{e22}$ has a unique solution, then these solutions are almost periodic.
\end{lemma}

\begin{definition}\label{def24}
Suppose that $\varphi(t)$ is any solution of $\eqref{e21}$ on $\mathbb{T}$. $\varphi(t)$ is said to be a strictly positive solution on $\mathbb{T}$ if for $t\in\mathbb{T}$,
\begin{eqnarray*}
0<\inf\limits_{t\in\mathbb{T}}\varphi(t)\leq\sup\limits_{t\in\mathbb{T}}\varphi(t)<\infty.
\end{eqnarray*}
\end{definition}

\begin{lemma}\label{lem211}
If each of the hull equations of system $\eqref{e21}$ has a unique strictly positive solution, then  system $\eqref{e21}$ has a unique strictly positive almost periodic solution.
\end{lemma}
\begin{proof}Suppose  that $\varphi(t)$ is a strictly positive solution of system \eqref{e22}.  Since $f$ is almost periodic in $t$ uniformly for $x\in S$, by Lemma \ref{lem28}, for any sequences $\alpha', \beta'\subset\Pi$, there exist common subsequences $\alpha\subset\alpha'$, $\beta\subset\beta'$ such that $T_{\alpha+\beta}f(t,x)=T_{\alpha}T_{\beta}f(t,x)$ holds uniformly in $t$ for $x\in S$, $T_{\alpha+\beta}\varphi(t)$  and $T_{\alpha}T_{\beta}\varphi(t)$ uniformly exist on compact set of $\mathbb{T}$. Then  $T_{\alpha+\beta}\varphi(t)$ and $T_{\alpha}T_{\beta}\varphi(t)$ are solutions of the equation
\begin{eqnarray*}
x^{\Delta}(t)=T_{\alpha+\beta}f(t,x),\,\,t\in\mathbb{T},
\end{eqnarray*}
which is the common hull equation of system \eqref{e21}, with respect to $\alpha$ and $\beta$, respectively.
Therefore, we have $T_{\alpha+\beta}\varphi(t)=T_{\alpha}T_{\beta}\varphi(t)$, then by  Lemma \ref{lem28},   $\varphi(t)$ is an almost periodic solution of \eqref{e21}. Since $\alpha\subset\alpha'\subset\Pi$ and $\lim\limits_{n\rightarrow\infty}\alpha'_{n}=+\infty$, $T_{\alpha}f(t,x)=g(t,x)$ exists uniformly in $t\in\mathbb{T}$ for $x\in S$. For the sequence $\alpha\subset\alpha'$, we conclude that $T_{\alpha}\varphi(t)=\psi(t)$ exists uniformly in $t\in\mathbb{T}$. According to the uniqueness of the solution and $T_{\alpha}\psi(t)=\psi(t)$, one obtains that $\varphi(t)=\psi(t)$.   The proof is completed.
\end{proof}

\begin{lemma}\label{lem212}\cite{20}
Assume that $a\in\mathcal{R}$ and $t_{0}\in{\mathbb{T}}$, if $a\in\mathcal{R}^{+}$ on $\mathbb{T}^{k}$, then $e_a(t,t_{0})> 0$ for all $t\in{\mathbb{T}}$.
\end{lemma}

\begin{lemma}\label{lem214}
Assume that $x(t)>0$ on $\mathbb{T}$, $-b\in\mathcal{R}^{+}$, $b\geq 0$, $a, d> 0$, $t-\tau(t)\in\mathbb{T}$, where $\tau:\mathbb{T}\rightarrow\mathbb{R}^{+}$ is a rd-continuous function and $\bar{\tau}=\sup\limits_{t\in \mathbb{T}}\{\tau(t)\}$.
\begin{itemize}
    \item  [$(i)$] If $x^{\Delta}(t)\leq x^{\sigma}(t)(b-ax(t-\tau(t)))+d$ for $t\geq t_{0}$, $t_{0}\in{\mathbb{T}}$, with the initial condition $x(t)=\phi(t)\geq0$ for $t\in[t_{0}-\bar{\tau}, t_{0}]_{\mathbb{T}}$ and $\phi(t_{0})>0$, then
        \begin{eqnarray*}
 \limsup_{t\rightarrow+\infty}x(t)\leq-\frac{d}{b}+\bigg(\frac{d}{b}+\bar{x}\bigg)\exp\bigg\{-\frac{\bar{\tau}\log(1-b\bar{\mu})}{\bar{\mu}}\bigg\}:=M,
   \end{eqnarray*}
where $\bar{\mu}=\sup\limits_{\theta\in \mathbb{T}}\{\mu(\theta)\}$ and $\bar{x}$ is the unique positive root of $x(ax-b)-d=0$.

Especially, if $d=0$, then
\begin{eqnarray*}
M=\frac{b}{a}\exp\bigg\{-\frac{\bar{\tau}\log(1-b\bar{\mu})}{\bar{\mu}}\bigg\}.
\end{eqnarray*}
    \item  [$(ii)$] If $x^{\Delta}(t)\geq x^{\sigma}(t)(b-ax(t-\tau(t)))+d$ for $t\geq t_{0}$, $t_{0}\in{\mathbb{T}}$, with the initial condition $x(t)=\phi(t)\geq0$ for $t\in[t_{0}-\bar{\tau}, t_{0}]_{\mathbb{T}}$, $\phi(t_{0})>0$ and there exists a positive constant $N>0$ such that $\limsup_{t\rightarrow+\infty}x(t)\leq N<+\infty$, then
   \begin{eqnarray*}
 \liminf_{t\rightarrow+\infty}x(t)\geq \frac{b}{a}e^{-aN\bar{\tau}}:=m.
   \end{eqnarray*}
\end{itemize}
\end{lemma}
\begin{proof}
The proof of (i). It is obviously that there exists a unique positive root of the equation $x(ax-b)-d=0$. Suppose that $\limsup\limits_{t\rightarrow+\infty}x(t)=+\infty$. Then there exists a subsequence $\{t_{k}\}_{k=1}^{\infty}\subset\mathbb{T}$ with $t_{k}\rightarrow+\infty$ as $k\rightarrow+\infty$ such that
 \begin{eqnarray*}
\lim\limits_{k\rightarrow+\infty}x(t_{k})=+\infty,\quad x^{\sigma}(t_{k})\geq \bar{x},\quad x^{\Delta}(t)|_{t=t_{k}}\geq0,\quad k=1, 2, \ldots.
\end{eqnarray*}
Thus, we have
\begin{eqnarray*}
x^{\sigma}(t_{k})(b-ax(t_{k}-\tau(t_k)))+d\geq 0,
\end{eqnarray*}
so
\begin{eqnarray}\label{e23}
x(t_{k}-\tau(t_k))\leq\frac{1}{a}\bigg(b+\frac{d}{x^{\sigma}(t_{k})}\bigg)\leq \frac{1}{a}\bigg(b+\frac{d}{\bar{x}}\bigg)=\bar{x},\,\,k=1, 2, \ldots.
\end{eqnarray}
Considering the following inequality
\begin{eqnarray*}
x^{\Delta}(t)\leq bx^{\sigma}(t)+d,\,\,\,{\rm with}\,\,x(t_{0}^{\ast})>0, t_{0}^{\ast}\geq t_{0}.
\end{eqnarray*}
Noticing that
\begin{eqnarray}\label{e24}
[x(t)e_{-b}(t, t_{0}^{\ast})]^{\Delta}&=&e_{-b}(t, t_{0}^{\ast})x^{\Delta}(t)-be_{-b}(t, t_{0}^{\ast})x^{\sigma}(t)\nonumber\\
&=&e_{-b}(t, t_{0}^{\ast})(x^{\Delta}(t)-bx^{\sigma}(t))\nonumber\\
&\leq&de_{-b}(t, t_{0}^{\ast}).
\end{eqnarray}
Integrating inequality \eqref{e24} from $t_{0}^{\ast}$ to $t$, we have
\begin{eqnarray*}
e_{-b}(t, t_{0}^{\ast})x(t)-x(t_{0}^{\ast})&\leq&\int_{t_{0}^{\ast}}^{t}de_{-b}(\theta, t_{0}^{\ast})\Delta\theta\nonumber\\
&=&-\frac{d}{b}\int_{t_{0}^{\ast}}^{t}[e_{-b}(\theta, t_{0}^{\ast})]^{\Delta}\Delta\theta\nonumber\\
&=&-\frac{d}{b}[e_{-b}(t, t_{0}^{\ast})-1],
\end{eqnarray*}
then
\begin{eqnarray}\label{e25}
x(t)\leq-\frac{d}{b}+\bigg(\frac{d}{b}+x(t_{0}^{\ast})\bigg)e_{\ominus(-b)}(t, t_{0}^{\ast}).
\end{eqnarray}
In view of \eqref{e23} and \eqref{e25}, we obtain
\begin{eqnarray}\label{e26}
x(t_{k})&\leq&-\frac{d}{b}+\bigg(\frac{d}{b}+x(t_{k}-\tau(t_k))\bigg)e_{\ominus(-b)}(t_{k}, t_{k}-\tau(t_k))\nonumber\\
&\leq&-\frac{d}{b}+\bigg(\frac{d}{b}+\bar{x}\bigg)e_{\ominus(-b)}(t_{k}, t_{k}-\tau(t_k)),\,\,k=1, 2, \ldots.
\end{eqnarray}
For every $\theta\in\mathbb{T}$, if $\mu(\theta)=0$, then
\begin{eqnarray*}
\xi_{\mu}(\ominus(-b))=\ominus(-b)=\frac{b}{1-\mu(\theta)b}=b,
\end{eqnarray*}
if $\mu(\theta)\neq0$, then
\begin{eqnarray*}
\xi_{\mu}(\ominus(-b))=\frac{\log\bigg(1+\frac{b\mu(\theta)}{1-\mu(\theta)b}\bigg)}{\mu(\theta)}
=-\frac{\log(1-b\mu(\theta))}{\mu(\theta)}\leq-\frac{\log(1-b\bar{\mu})}{\bar{\mu}}(>b).
\end{eqnarray*}
Hence, for every $\theta\in\mathbb{T}$, we have
\begin{eqnarray*}
\int_{t_{k}-\tau(t_k)}^{t_{k}}\xi_{\mu}(\ominus(-b))\Delta\theta&\leq&\max\bigg\{\int_{t_{k}-\tau(t_k)}^{t_{k}}b\Delta\theta, \int_{t_{k}-\tau(t_k)}^{t_{k}}-\frac{\log(1-b\bar{\mu})}{\bar{\mu}}\Delta\theta\bigg\}\\
&=& -\frac{\tau(t_k)\log(1-b\bar{\mu})}{\bar{\mu}}\\
&\leq&-\frac{\bar{\tau}\log(1-b\bar{\mu})}{\bar{\mu}},\,\,k=1, 2, \ldots.
\end{eqnarray*}
Thus
\begin{eqnarray}\label{e27}
e_{\ominus(-b)}(t_{k}, t_{k}-\tau(t_k))\leq \exp\bigg\{-\frac{\bar{\tau}\log(1-b\bar{\mu})}{\bar{\mu}}\bigg\},\,\,k=1, 2, \ldots.
\end{eqnarray}
It follows from $\eqref{e26}$ and $\eqref{e27}$ that
\begin{eqnarray*}
x(t_{k})\leq-\frac{d}{b}+\bigg(\frac{d}{b}+\bar{x}\bigg)
\exp\bigg\{-\frac{\bar{\tau}\log(1-b\bar{\mu})}{\bar{\mu}}\bigg\}=M,\,\,k=1, 2, \ldots,
\end{eqnarray*}
then
\begin{eqnarray*}
 \limsup_{k\rightarrow+\infty}x(t_{k})\leq M.
   \end{eqnarray*}
Especially, if $d=0$, then $\bar{x}=\frac{b}{a}$, we can easily know that
    \begin{eqnarray*}
M=\frac{b}{a}\exp\bigg\{-\frac{\bar{\tau}\log(1-b\bar{\mu})}{\bar{\mu}}\bigg\}.
\end{eqnarray*}
Hence $\limsup\limits_{k\rightarrow+\infty}x(t_{k})<+\infty$. This contradicts the assumption.

We claim
\begin{eqnarray*}
 \limsup_{t\rightarrow+\infty}x(t)\leq M.
   \end{eqnarray*}
Otherwise,
\begin{eqnarray*}
 \limsup_{t\rightarrow+\infty}x(t)> M,
   \end{eqnarray*}
 there exists $\varepsilon$ such that $x(t)> M+\varepsilon$ for any $t\in\mathbb{T}$.
So we can choose $\{t_{k}\}_{k=1}^{\infty}\subset\mathbb{T}$ such that
\begin{eqnarray*}
x(t_{k})> M+\varepsilon,\quad x^{\sigma}(t_{k})\geq\bar{x},\quad x^{\Delta}(t)|_{t=t_{k}}\geq0,\quad k=1, 2, \ldots.
   \end{eqnarray*}
By a similar process as above, we can derive that
\begin{eqnarray*}
x(t_{k})\leq M,
   \end{eqnarray*}
which is a contradiction. Hence, our claim holds.

The proof of (ii). Suppose that $\liminf\limits_{t\rightarrow+\infty}x(t)=0$. Then there exists a subsequence $\{\tilde{t_{k}}\}_{1}^{\infty}\subset\mathbb{T}$ with $\tilde{t_{k}}\rightarrow+\infty$ as $k\rightarrow+\infty$ such that
 \begin{eqnarray*}
\lim_{k\rightarrow+\infty}x(\tilde{t_{k}})=0,\quad x^{\Delta}(t)|_{t=\tilde{t_{k}}}\leq0,\quad k=1, 2, \ldots.
\end{eqnarray*}
We have $b-ax(\tilde{t_{k}}-\tau(\tilde{t_{k}}))\leq0$, then $x(\tilde{t_{k}}-\tau(\tilde{t_{k}}))\geq\frac{b}{a}$.
For any positive constant $\varepsilon$ small enough, it follows from $\limsup_{t\rightarrow+\infty}x(t)\leq N$ that there exists large enough $T_{1}$
such that
\begin{eqnarray*}
 x(t)\leq N+\varepsilon,\,\, t>T_{1},
\end{eqnarray*}
then $x(\tilde{t_{k}}-\tau(\tilde{t_{k}}))\leq N+\varepsilon$ for $\tilde{t_{k}}>T_{1}+\tau(\tilde{t_{k}})$. So we have
\begin{eqnarray}\label{e28}
 x^{\Delta}(\tilde{t_{k}})&\geq& x^{\sigma}(\tilde{t_{k}})(b-ax(\tilde{t_{k}}-\tau(\tilde{t_{k}})))+d\nonumber\\
 &\geq& x^{\sigma}(\tilde{t_{k}})(b-ax(\tilde{t_{k}}-\tau(\tilde{t_{k}})))\nonumber\\
 &\geq& -a(N+\varepsilon)x^{\sigma}(\tilde{t_{k}}),\quad k=1, 2, \ldots.
\end{eqnarray}

Considering the following inequality
\begin{eqnarray*}
x^{\Delta}(t)\geq -a(N+\varepsilon)x^{\sigma}(t),\,\,\,{\rm with}\,\,x(t_{0}^{\ast})>0, t_{0}^{\ast}\geq t_{0}.
\end{eqnarray*}
For $t>t_{0}^{\ast}\geq t_{0}$, we have
\begin{eqnarray}\label{e29}
 x(t)\geq x(t_{0}^{\ast})e_{\ominus(a(N+\varepsilon))}(t, t_{0}^{\ast}).
   \end{eqnarray}
From $\eqref{e28}$ and $\eqref{e29}$, we obtain
\begin{eqnarray}\label{e210}
 x(\tilde{t_{k}})\geq x(\tilde{t_{k}}-\tau(\tilde{t_{k}}))e_{\ominus(a(N+\varepsilon))}(\tilde{t_{k}}, \tilde{t_{k}}-\tau(\tilde{t_{k}})).
   \end{eqnarray}
For every $\theta\in\mathbb{T}$, if $\mu(\theta)=0$, then
\begin{eqnarray*}
\xi_{\mu}(\ominus(a(N+\varepsilon)))=\ominus(a(N+\varepsilon))=-\frac{a(N+\varepsilon)}{1+\mu(\theta)a(N+\varepsilon)}=-a(N+\varepsilon),
\end{eqnarray*}
if $\mu(\theta)\neq0$, then
\begin{eqnarray*}
\xi_{\mu}(\ominus(a(N+\varepsilon)))=\frac{\log\bigg(1-\frac{a(N+\varepsilon)\mu(\theta)}{1+\mu(\theta)a(N+\varepsilon)}\bigg)}{\mu(\theta)}
=-\frac{\log(1+a(N+\varepsilon)\mu(\theta))}{\mu(\theta)}\geq-a(N+\varepsilon).
\end{eqnarray*}

Hence, for every $\theta\in\mathbb{T}$, we have
\begin{eqnarray*}
\int_{\tilde{t_{k}}-\tau(\tilde{t_k)}}^{\tilde{t_{k}}}\xi_{\mu}(\ominus(a(N+\varepsilon)))\Delta\theta
&\geq&\int_{\tilde{t_{k}}-\tau(\tilde{t_{k}})}^{\tilde{t_{k}}}-a(N+\varepsilon)\Delta\theta\\
&=&-a(N+\varepsilon)\tau(\tilde{t_{k}})\\
&\geq&-a(N+\varepsilon)\bar{\tau},\,\,k=1, 2, \ldots,
\end{eqnarray*}
so
\begin{eqnarray*}
\exp\bigg\{\int_{\tilde{t_{k}}-\tau(\tilde{t_k)}}^{\tilde{t_{k}}}\xi_{\mu}(\ominus(a(N+\varepsilon)))\bigg\}\Delta\theta
\geq e^{-a(N+\varepsilon)\bar{\tau}},\,\,k=1, 2, \ldots.
\end{eqnarray*}
Thus
\begin{eqnarray}\label{e211}
e_{\ominus(a(N+\varepsilon))}(\tilde{t_{k}}, \tilde{t_{k}}-\tau(\tilde{t_k}))\geq e^{-a(N+\varepsilon)\bar{\tau}},\,\,k=1, 2, \ldots.
\end{eqnarray}
By use of $\eqref{e210}$ and $\eqref{e211}$, we obtain
\begin{eqnarray*}
 x(\tilde{t_{k}})\geq x(\tilde{t_{k}}-\tau(\tilde{t_{k}}))e^{-a(N+\varepsilon)\bar{\tau}}\geq\frac{b}{a}e^{-a(N+\varepsilon)\bar{\tau}},\,\,k=1, 2, \ldots.
   \end{eqnarray*}
Letting $\varepsilon\rightarrow0$, then
\begin{eqnarray*}
 \liminf_{k\rightarrow+\infty}x(\tilde{t_{k}})\geq\frac{b}{a}e^{-aN\bar{\tau}}=m,\,\,k=1, 2, \ldots.
   \end{eqnarray*}

Similarly, we can get
\begin{eqnarray*}
 \liminf_{t\rightarrow+\infty}x(t)\geq m.
   \end{eqnarray*}
The proof of Lemma \ref{lem214} is completed.
\end{proof}

Similar to the proof of Lemma \ref{lem214}, we can easily obtain the following results:
\begin{lemma}\label{lem215}
Assume that $x(t)>0$ on $\mathbb{T}, b\geq0, a, d> 0, t-\tau(t)\in\mathbb{T}$, where $\tau(t):\mathbb{T}\rightarrow\mathbb{R}^{+}$ is a rd-continuous function  and $\bar{\tau}=\sup\limits_{t\in \mathbb{T}}\{\tau(t)\}$.
\begin{itemize}
    \item  [$(i)$] If $x^{\Delta}(t)\leq x(t)(b-ax(t-\tau(t)))+d$ for $t\geq t_{0}$, $t_{0}\in{\mathbb{T}}$, with initial condition $x(t)=\phi(t)\geq0$ for $t\in[t_{0}-\tau, t_{0}]_{\mathbb{T}}$, $\phi(t_{0})>0$, then
    \begin{eqnarray*}
 \limsup_{t\rightarrow+\infty}x(t)\leq -\frac{d}{b}+\bigg(\frac{d}{b}+\bar{x}\bigg)e^{b\bar{\tau}}:=\tilde{M},
   \end{eqnarray*}
where $\bar{x}$ is the unique positive root of $x(ax-b)-d=0$.

Especially, if $d=0$, then
    \begin{eqnarray*}
\tilde{M}=\frac{b}{a}e^{b\bar{\tau}}.
\end{eqnarray*}
    \item  [$(ii)$] If $x^{\Delta}(t)\geq x(t)(b-ax(t-\tau(t)))+d$ for $t\geq t_{0}$, $t_{0}\in{\mathbb{T}}$, with initial condition $x(t)=\phi(t)\geq0$ for $t\in[t_{0}-\tau, t_{0}]_{\mathbb{T}}$, $\phi(t_{0})>0$ and there exists a positive constant $\tilde{N}>0$ such that $\limsup_{t\rightarrow+\infty}x(t)\leq \tilde{N}<+\infty$ and $-a\tilde{N}\in\mathcal{R}^{+}$, then
    \begin{eqnarray*}
 \liminf_{t\rightarrow+\infty}x(t)\geq \frac{b}{a}\exp\bigg\{\frac{\bar{\tau}\log(1-a\tilde{N}\bar{\mu})}{\bar{\mu}}\bigg\}:=\tilde{m},
   \end{eqnarray*}
where $\bar{\mu}=\sup\limits_{\theta\in \mathbb{T}}\{\mu(\theta)\}$.
\end{itemize}
\end{lemma}
\begin{proof}
The proof of $(i)$.  Suppose that $\limsup\limits_{t\rightarrow+\infty}x(t)=+\infty$. Then there exists a subsequence $\{t_{k}\}_{k=1}^{\infty}\subset\mathbb{T}$ with $t_{k}\rightarrow+\infty$ as $k\rightarrow+\infty$ such that
 \begin{eqnarray*}
\lim\limits_{k\rightarrow+\infty}x(t_{k})=+\infty,\quad x(t_{k})\geq \bar{x},\quad x^{\Delta}(t)|_{t=t_{k}}\geq0,\quad k=1, 2, \ldots.
\end{eqnarray*}
Thus, we have
\begin{eqnarray*}
x(t_{k})(b-ax(t_{k}-\tau(t_k)))+d\geq 0,
\end{eqnarray*}
so
\begin{eqnarray}\label{e212}
x(t_{k}-\tau(t_k))\leq\frac{1}{a}\bigg(b+\frac{d}{x(t_{k})}\bigg)\leq \frac{1}{a}\bigg(b+\frac{d}{\bar{x}}\bigg)=\bar{x},\,\,k=1, 2, \ldots.
\end{eqnarray}
Considering the following inequality
\begin{eqnarray*}
x^{\Delta}(t)\leq bx(t)+d,\,\,\,{\rm with}\,\,x(t_{0}^{\ast})>0, t_{0}^{\ast}\geq t_{0}.
\end{eqnarray*}
Noticing that
\begin{eqnarray}\label{e213}
[x(t)e_{\ominus b}(t, t_{0}^{\ast})]^{\Delta}&=&e_{\ominus b}(\sigma(t), t_{0}^{\ast})x^{\Delta}(t)-be_{\ominus b}(\sigma(t), t_{0}^{\ast})x(t)\nonumber\\
&=&e_{\ominus b}(\sigma(t), t_{0}^{\ast})(x^{\Delta}(t)-bx(t))\nonumber\\
&\leq&de_{\ominus b}(t, t_{0}^{\ast}).
\end{eqnarray}
Integrating inequality \eqref{e213} from $t_{0}^{\ast}$ to $t$, we have
\begin{eqnarray*}
e_{\ominus b}(t, t_{0}^{\ast})x(t)-x(t_{0}^{\ast})&\leq&\int_{t_{0}^{\ast}}^{t}de_{\ominus b}(\sigma(\tau), t_{0}^{\ast})\Delta\tau\nonumber\\
&=&-\frac{d}{b}\int_{t_{0}^{\ast}}^{t}[e_{\ominus b}(\tau, t_{0}^{\ast})]^{\Delta}\Delta\tau\nonumber\\
&=&-\frac{d}{b}[e_{\ominus b}(t, t_{0}^{\ast})-1],
\end{eqnarray*}
then
\begin{eqnarray}\label{e214}
x(t)\leq-\frac{d}{b}+\bigg(\frac{d}{b}+x(t_{0}^{\ast})\bigg)e_{b}(t, t_{0}^{\ast}).
\end{eqnarray}
In view of \eqref{e212} and \eqref{e214}, we obtain
\begin{eqnarray}\label{e215}
x(t_{k})&\leq&-\frac{d}{b}+\bigg(\frac{d}{b}+x(t_{k}-\tau(t_{k}))\bigg)e_{b}(t_{k}, t_{k}-\tau(t_{k}))\nonumber\\
&\leq&-\frac{d}{b}+\bigg(\frac{d}{b}+\bar{x}\bigg)e_{b}(t_{k}, t_{k}-\tau(t_{k})),\,\,k=1, 2, \ldots.
\end{eqnarray}
For every $\theta\in\mathbb{T}$, if $\mu(\theta)=0$, then
\begin{eqnarray*}
\xi_{\mu}(b)=b,
\end{eqnarray*}
if $\mu(\theta)\neq0$, then
\begin{eqnarray*}
\xi_{\mu}(b)=\frac{\log(1+b\mu(\theta))}{\mu(\theta)}\leq b.
\end{eqnarray*}
Hence, for every $\theta\in\mathbb{T}$, we have
\begin{eqnarray*}
\int_{t_{k}-\tau(t_k)}^{t_{k}}\xi_{\mu}(b)\Delta\theta&\leq&\int_{t_{k}-\tau(t_k)}^{t_{k}}b\Delta\theta\leq b\tau(t_k)\leq b\bar{\tau},\,\,k=1, 2, \ldots,
\end{eqnarray*}
so
\begin{eqnarray*}
\exp\bigg\{\int_{t_{k}-\tau(t_k)}^{t_{k}}\xi_{\mu}(b)\bigg\}\Delta\theta\leq e^{b\bar{\tau}},\,\,k=1, 2, \ldots.
\end{eqnarray*}
Thus
\begin{eqnarray}\label{e216}
e_{b}(t_{k}, t_{k}-\tau(t_k))\leq e^{b\bar{\tau}},\,\,k=1, 2, \ldots.
\end{eqnarray}
It follows from $\eqref{e215}$ and $\eqref{e216}$ that
\begin{eqnarray*}
x(t_{k})\leq-\frac{d}{b}+\bigg(\frac{d}{b}+\bar{x}\bigg)e^{b\bar{\tau}}=\tilde{M},\,\,k=1, 2, \ldots,
\end{eqnarray*}
then
\begin{eqnarray*}
 \limsup_{k\rightarrow+\infty}x(t_{k})\leq \tilde{M}.
   \end{eqnarray*}
Especially, if $d=0$, then $\bar{x}=\frac{b}{a}$, we can easily know that
    \begin{eqnarray*}
\tilde{M}=\frac{b}{a}e^{b\bar{\tau}}.
\end{eqnarray*}
Hence $\limsup\limits_{k\rightarrow+\infty}x(t_{k})<+\infty$. This contradicts the assumption.

Similarly, we can get
\begin{eqnarray*}
 \limsup_{t\rightarrow+\infty}x(t)\leq \tilde{M}.
   \end{eqnarray*}

The proof of (ii). Suppose that $\liminf\limits_{t\rightarrow+\infty}x(t)=0$. Then there exists a subsequence $\{\tilde{t_{k}}\}_{1}^{\infty}\subset\mathbb{T}$ with $\tilde{t_{k}}\rightarrow+\infty$ as $k\rightarrow+\infty$ such that
 \begin{eqnarray*}
\lim_{k\rightarrow+\infty}x(\tilde{t_{k}})=0,\quad x^{\Delta}(t)|_{t=\tilde{t_{k}}}\leq0,\quad k=1, 2, \ldots.
\end{eqnarray*}
We have $b-ax(\tilde{t_{k}}-\tau(\tilde{t_{k}}))\leq0$, then $x(\tilde{t_{k}}-\tau(\tilde{t_{k}}))\geq\frac{b}{a}$.
For any positive constant $\varepsilon$ small enough, it follows from $\limsup_{t\rightarrow+\infty}x(t)\leq \tilde{N}$ that there exists large enough $T_{2}$
such that
\begin{eqnarray*}
 x(t)\leq \tilde{N}+\varepsilon,\,\, t>T_{2},
\end{eqnarray*}
then $x(\tilde{t_{k}}-\tau(\tilde{t_{k}}))\leq \tilde{N}+\varepsilon$ for $\tilde{t_{k}}>T_{2}+\tau(\tilde{t_{k}})$ and $b-a(\tilde{N}+\varepsilon)\leq0$. So we have
\begin{eqnarray}\label{e217}
 x^{\Delta}(\tilde{t_{k}})&\geq& x(\tilde{t_{k}})(b-ax(\tilde{t_{k}}-\tau(\tilde{t_{k}})))+d\nonumber\\
 &\geq& x(\tilde{t_{k}})(b-ax(\tilde{t_{k}}-\tau(\tilde{t_{k}})))\nonumber\\
 &\geq& -a(\tilde{N}+\varepsilon)x(\tilde{t_{k}}),\quad k=1, 2, \ldots.
\end{eqnarray}

Considering the following inequality
\begin{eqnarray*}
x^{\Delta}(t)\geq -a(\tilde{N}+\varepsilon)x(t),\,\,\,{\rm with}\,\,x(t_{0}^{\ast})>0,\,\, t_{0}^{\ast}\geq t_{0}.
\end{eqnarray*}
For $t>t_{0}^{\ast}\geq t_{0}$, we have
\begin{eqnarray}\label{e218}
 x(t)\geq x(t_{0}^{\ast})e_{-a(\tilde{N}+\varepsilon)}(t, t_{0}^{\ast}).
   \end{eqnarray}
From $\eqref{e217}$ and $\eqref{e218}$, we obtain
\begin{eqnarray}\label{e219}
 x(\tilde{t_{k}})\geq x(\tilde{t_{k}}-\tau(\tilde{t_{k}}))e_{-a(N+\varepsilon)}(\tilde{t_{k}}, \tilde{t_{k}}-\tau(\tilde{t_{k}})).
   \end{eqnarray}
For every $\theta\in\mathbb{T}$, if $\mu(\theta)=0$, then
\begin{eqnarray*}
\xi_{\mu}(-a(\tilde{N}+\varepsilon))=-a(\tilde{N}+\varepsilon),
\end{eqnarray*}
if $\mu(\theta)\neq0$, then
\begin{eqnarray*}
\xi_{\mu}(-a(\tilde{N}+\varepsilon))=\frac{\log(1-a(\tilde{N}
+\varepsilon)\mu(\theta))}{\mu(\theta)}
\geq\frac{\log(1-a(\tilde{N}+\varepsilon)\bar{\mu})}{\bar{\mu}}<-a(\tilde{N}+\varepsilon).
\end{eqnarray*}

Hence, for every $\theta\in\mathbb{T}$, we have
\begin{eqnarray*}
\int_{\tilde{t_{k}}-\tau(\tilde{t_{k}})}^{\tilde{t_{k}}}\xi_{\mu}(-a(\tilde{N}+\varepsilon))\Delta\theta
&\geq&\min\bigg\{\int_{\tilde{t_{k}}-\tau(\tilde{t_{k}})}^{\tilde{t_{k}}}-a(\tilde{N}+\varepsilon)\Delta\theta, \int_{\tilde{t_{k}}-\tau(\tilde{t_{k}})}^{\tilde{t_{k}}}\frac{\log(1-a(\tilde{N}+\varepsilon)\bar{\mu})}{\bar{\mu}}\Delta\theta\bigg\}\\
&=& \frac{\tau(\tilde{t_{k}})\log(1-a(\tilde{N}+\varepsilon)\bar{\mu})}{\bar{\mu}}\\
&\geq&\frac{\bar{\tau}\log(1-a(\tilde{N}+\varepsilon)\bar{\mu})}{\bar{\mu}},\,\,k=1, 2, \ldots,
\end{eqnarray*}
so
\begin{eqnarray*}
\exp\bigg\{\int_{\tilde{t_{k}}-\tau(\tilde{t_k)}}^{\tilde{t_{k}}}\xi_{\mu}(-a(\tilde{N}+\varepsilon))\bigg\}\Delta\theta
\geq \exp\bigg\{\frac{\bar{\tau}\log(1-a(\tilde{N}+\varepsilon)\bar{\mu})}{\bar{\mu}}\bigg\},\,\,k=1, 2, \ldots.
\end{eqnarray*}
Thus
\begin{eqnarray}\label{e220}
e_{-a(\tilde{N}+\varepsilon)}(\tilde{t_{k}}, \tilde{t_{k}}-\tau(\tilde{t_k}))\geq \exp\bigg\{\frac{\bar{\tau}\log(1-a(\tilde{N}+\varepsilon)\bar{\mu})}{\bar{\mu}}\bigg\},\,\,k=1, 2, \ldots.
\end{eqnarray}
By use of $\eqref{e219}$ and $\eqref{e220}$, we obtain
\begin{eqnarray*}
 x(\tilde{t_{k}})&\geq& x(\tilde{t_{k}}-\tau(\tilde{t_{k}}))\exp\bigg\{\frac{\bar{\tau}\log(1-a(\tilde{N}+\varepsilon)\bar{\mu})}{\bar{\mu}}\bigg\}\\
 &\geq&\frac{b}{a}\exp\bigg\{\frac{\bar{\tau}\log(1-a(\tilde{N}+\varepsilon)\bar{\mu})}{\bar{\mu}}\bigg\},\,\,k=1, 2, \ldots.
   \end{eqnarray*}
Letting $\varepsilon\rightarrow0$, we have
\begin{eqnarray*}
 \liminf_{k\rightarrow+\infty}x(\tilde{t_{k}})\geq\frac{b}{a}\exp\bigg\{\frac{\bar{\tau}\log(1-a\tilde{N}\bar{\mu})}{\bar{\mu}}\bigg\}=\tilde{m},\,\,k=1, 2, \ldots.
   \end{eqnarray*}

Similarly, we can get
\begin{eqnarray*}
 \liminf_{t\rightarrow+\infty}x(t)\geq \tilde{m}.
   \end{eqnarray*}
The proof of Lemma \ref{lem215} is completed.
\end{proof}

\section{Permanence}

\setcounter{equation}{0}
{\setlength\arraycolsep{2pt}}
 \indent

In this section, we will give our main results about the permanence of system $\eqref{e13}$. For convenience, we introduce the following notations:
\begin{eqnarray*}
&&x_{i}^{M}=\ln\bigg\{\frac{ a_{i}^{u}+\sum_{j=1,j\neq i}^{n}c_{ij}^{u}}{b_{i}^{l}}\exp\bigg\{-\frac{\tau^{+}\log(1-(a_{i}^{u}+\sum_{j=1,j\neq i}^{n}c_{ij}^{u})\bar{\mu})}{\bar{\mu}}\bigg\}\bigg\},\\
&&x_{i}^{m}=\ln\bigg\{\frac{a_{i}^{l}}{b_{i}^{u}}\exp\bigg\{\frac{\tau^{+}\log(1-b_{i}^{u}e^{x_{i}^{M}}\bar{\mu})}{\bar{\mu}}\bigg\}\bigg\},\,\,i=1, 2,\ldots,n,
\end{eqnarray*}
where $\bar{\mu}=\sup\limits_{t\in \mathbb{T}}\{\mu(t)\}$.
\begin{enumerate}
  \item [$(H_{3})$]
  $a_{i}^{l}\exp\bigg\{\frac{\tau^{+}\log(1-b_{i}^{u}e^{x_{i}^{M}}\bar{\mu})}{\bar{\mu}}\bigg\}>b_{i}^{u}$, $ -(a_{i}^{u}+\sum_{j=1,j\neq i}^{n}c_{ij}^{u})\in\mathcal{R}^{+}$ and  $-b_{i}^{u}e^{x_{i}^{M}}\in\mathcal{R}^{+}$, $i=1, 2,\ldots,n$.
\end{enumerate}

\begin{lemma}\label{lem31}
Assume that $(H_{1})-(H_{3})$ hold. Let $x(t)=(x_{1}(t), x_{2}(t)),\ldots, x_{n}(t))$ be any solution of system $\eqref{e13}$ with initial condition $\eqref{e14}$, then
 \begin{eqnarray*}
x_i^{m}\leq\liminf_{t\rightarrow+\infty}x_i(t)\leq\limsup_{t\rightarrow+\infty}x_i(t)\leq x_i^{M},\,\,i=1, 2,\ldots,n.
 \end{eqnarray*}
\end{lemma}
\begin{proof}
Let $x(t)=(x_{1}(t), x_{2}(t)),\ldots, x_{n}(t))$ be any solution of system $\eqref{e13}$ with initial condition $\eqref{e14}$. From \eqref{e13} it follows that
\begin{eqnarray*}\label{e31}
x_{i}^{\Delta}(t) \leq a_{i}^{u}+\sum_{j=1,j\neq i}^{n}c_{ij}^{u}-b_{i}^{l}e^{x_{i}(t-\tau_{i}(t))},\,\,i=1,2,\ldots, n.
\end{eqnarray*}
Let $N_{i}(t)=e^{x_{i}(t)}$, obviously $N_{i}(t)>0$, the above inequality yields that
\begin{eqnarray*}
[\ln(N_{i}(t))]^{\Delta}\leq a_{i}^{u}+\sum_{j=1,j\neq i}^{n}c_{ij}^{u}-b_{i}^{l}N_{i}(t-\tau_{i}(t)), \,\,i=1, 2,\ldots,n.
\end{eqnarray*}
In view of Lemma 2.4, we have
\begin{eqnarray*}
\frac{N_{i}^{\Delta}(t)}{N_{i}(\sigma(t))}\leq a_{i}^{u}+\sum_{j=1,j\neq i}^{n}c_{ij}^{u}-b_{i}^{l}N_{i}(t-\tau_{i}(t)),
\end{eqnarray*}
 then
\begin{eqnarray*}
N_{i}^{\Delta}(t)\leq N_{i}(\sigma(t))\bigg[ a_{i}^{u}+\sum_{j=1,j\neq i}^{n}c_{ij}^{u}-b_{i}^{l}N_{i}(t-\tau_{i}(t))\bigg], \,\,i=1, 2,\ldots,n.
\end{eqnarray*}
By applying Lemma 2.14, there exists a constant $T_{0}$ such that
\begin{eqnarray*}
N_{i}(t)\leq \frac{ a_{i}^{u}+\sum_{j=1,j\neq i}^{n}c_{ij}^{u}}{b_{i}^{l}}\exp\bigg\{-\frac{\tau^{+}\log(1-(a_{i}^{u}+\sum_{j=1,j\neq i}^{n}c_{ij}^{u})\bar{\mu})}{\bar{\mu}}\bigg\}
\end{eqnarray*}
for $t\geq T_{0}+\tau^{+}$. That is, for $i=1, 2,\ldots,n$,
\begin{eqnarray*}
\limsup_{t\rightarrow+\infty}x_{i}(t)\leq \ln\bigg\{\frac{ a_{i}^{u}+\sum_{j=1,j\neq i}^{n}c_{ij}^{u}}{b_{i}^{l}}\exp\bigg\{-\frac{\tau^{+}\log(1-(a_{i}^{u}+\sum_{j=1,j\neq i}^{n}c_{ij}^{u})\bar{\mu})}{\bar{\mu}}\bigg\}\bigg\}=x_{i}^{M}.
\end{eqnarray*}

On the other hand, from $\eqref{e13}$  it follows that
\begin{eqnarray*}
x_{i}^{\Delta}(t) &\geq& a_{i}^{l}-b_{i}^{u}e^{x_{i}(t-\tau_{i}(t))},\,\,i=1,2,\ldots, n.
\end{eqnarray*}
 Let $N_{i}(t)=e^{x_{i}(t)}$, obviously $N_{i}(t)>0$, then the above inequality yields that
\begin{eqnarray*}
[\ln(N_{i}(t))]^{\Delta}\geq a_{i}^{l}-b_{i}^{u}N_{i}(t-\tau_{i}(t)).
\end{eqnarray*}
In view of Lemma \ref{lem24}, we have
\begin{eqnarray*}
\frac{N_{i}^{\Delta}(t)}{N_{i}(t)}\geq a_{i}^{l}-b_{i}^{u}N_{i}(t-\tau_{i}(t)),
\end{eqnarray*}
 then
\begin{eqnarray*}
N_{i}^{\Delta}(t)\geq N_{i}(t)[ a_{i}^{l}-b_{i}^{u}N_{i}(t-\tau_{i}(t))],\,\,i=1,2,\ldots, n.
\end{eqnarray*}
By applying Lemma \ref{lem215} and $a_{i}^{l}\exp\bigg\{\frac{\tau^{+}\log(1-b_{i}^{u}e^{x_{i}^{M}}\bar{\mu})}{\bar{\mu}}\bigg\}>b_{i}^{u}$, there exists a constant $T_{1}$ such that
\begin{eqnarray*}
N_{i}(t)\geq \frac{a_{i}^{l}}{b_{i}^{u}}\exp\bigg\{\frac{\tau^{+}\log(1-b_{i}^{u}e^{x_{i}^{M}}\bar{\mu})}{\bar{\mu}}\bigg\},\,\,i=1,2,\ldots, n
\end{eqnarray*}
for $t\geq T_{1}+\tau^{+}$. Therefore,
\begin{eqnarray*}
\liminf_{t\rightarrow+\infty}x_{i}(t)
\geq \ln\bigg\{\frac{a_{i}^{l}}{b_{i}^{u}}\exp\bigg\{\frac{\tau^{+}\log(1-b_{i}^{u}e^{x_{i}^{M}}\bar{\mu})}{\bar{\mu}}\bigg\}\bigg\}
=x_{i}^{m},\,\,i=1,2,\ldots, n.
\end{eqnarray*}
The proof is complete.
\end{proof}

\begin{theorem}\label{thm31}
Assume that $(H_{1})-(H_{3})$ hold, then system $\eqref{e13}$ with initial condition $\eqref{e14}$ is permanence.
\end{theorem}

\section{Global attractivity}

\setcounter{equation}{0}
{\setlength\arraycolsep{2pt}}
 \indent

In this section, we will study the global attractivity of system $\eqref{e13}$.

\begin{definition}\label{def41}
System $\eqref{e13}$ is said to be globally attractive if any two positive solutions $x(t)=(x_{1}(t), x_{2}(t),\ldots,x_{n}(t))$ with initial value $\varphi(s)=(\varphi_{1}(s),\varphi_2(s)\ldots,\varphi_n(s))$ and $y(t)=(y_{1}(t), y_{2}(t),\ldots,y_{n}(t))$ with initial value $\psi(s)=(\psi_{1}(s),\psi_2(s),\ldots,\psi_n(s))$ of system $\eqref{e13}$   satisfy
\begin{eqnarray*}
 \lim_{t\rightarrow\infty}|x_{i}(t)-y_{i}(t)|=0,\,\,i=1, 2,\ldots,n.
\end{eqnarray*}
\end{definition}

\begin{theorem}\label{thm41}
Assume that $(H_{1})-(H_{3})$ hold.
Suppose further that
\begin{itemize}
 \item  [$(H_{4})$] $\gamma_{i}>0$, where $i=1, 2,\ldots,n$,
\begin{eqnarray*}
\gamma_{i}&=&b_{i}^{l}e^{x_{i}^{m}}-2\bar{\mu}(b_{i}^{u}e^{x_{i}^{M}})^{2}-\frac{(b_{i}^{u}e^{x_{i}^{M}})^{2}[2\bar{\mu} b_{i}^{u}e^{x_{i}^{M}}+1](2\tau^{+}-\tau^{-})}{1-\tau^{\Delta}}\nonumber\\
&&-{\displaystyle\sum_{j=1,j\neq i}^{n}\frac{(c_{ij}^{u}e^{x_{j}^{M}})^{2}(2\bar{\mu} b_{i}^{u}e^{x_{i}^{M}}+1)(2\delta^{+}-\delta^{-})}{(d_{ij}+e^{x_{j}^{m}})^{4}(1-\delta^{\Delta})}}-{\displaystyle\sum_{j=1,j\neq i}^{n}\frac{c_{ji}^{u}e^{x_{i}^{M}}(2\bar{\mu} b_{j}^{u}e^{x_{j}^{M}}+1)}{(d_{ji}+e^{x_{i}^{m}})^{2}}}\nonumber\\
&&\times\bigg[1+\frac{b_{j}^{u}e^{x_{j}^{M}}(\tau^{+}+\delta^{+}-\delta^{-})}
{1-\delta^{\Delta}}+\frac{e^{x_{i}^{M}}b_{i}^{u}(\tau^{+}+\delta^{+}-\tau^{-})}{1-\tau^{\Delta}}\bigg],
\end{eqnarray*}
where $x_{i}^{m}$, $x_{i}^{M}$ defined in Lemma \ref{lem31} and $\bar{\mu}=\sup\limits_{t\in \mathbb{T}}\{\mu(t)\}$.
\end{itemize}
Then system $\eqref{e13}$ is globally attractive.
\end{theorem}
\begin{proof}
Assume that $x(t)=(x_{1}(t), x_{2}(t),\ldots,x_{n}(t))$ and $y(t)=(y_{1}(t), y_{2}(t),\ldots,y_{n}(t))$ are any solutions of system $\eqref{e13}$  with the initial values $\varphi(s)=(\varphi_{1}(s),\varphi_2(s),\ldots,\varphi_n(s))$ and $\psi(s)=(\psi_{1}(s),\psi_2(s),\ldots,\psi_n(s))$, respectively. In view of system $\eqref{e13}$, we have{\setlength\arraycolsep{2pt}\begin{eqnarray*}
\left\{
\begin{array}{lll}
x_{i}^{\Delta}(t)=a_{i}(t)-b_{i}(t)e^{x_{i}(t-\tau_{i}(t))}+\sum_{j=1,j\neq i}^{n}c_{ij}(t)\displaystyle\frac{e^{x_{j}(t-\delta_{j}(t))}}{d_{ij}+e^{x_{j}(t-\delta_{j}(t))}},\,\,i=1,2,\ldots, n,\\
y_{i}^{\Delta}(t)=a_{i}(t)-b_{i}(t)e^{y_{i}(t-\tau_{i}(t))}+\sum_{j=1,j\neq i}^{n}c_{ij}(t)\displaystyle\frac{e^{y_{j}(t-\delta_{j}(t))}}{d_{ij}+e^{y_{j}(t-\delta_{j}(t))}},\,\,i=1,2,\ldots, n,\\
\end{array}
\right.
\end{eqnarray*}}
then
{\setlength\arraycolsep{2pt}\begin{eqnarray}\label{e41}
(x_{i}(t)-y_{i}(t))^{\Delta}&=& -b_{i}(t)\big[e^{x_{i}(t-\tau_{i}(t))}-e^{y_{i}(t-\tau_{i}(t))}\big]\nonumber\\
&&+\sum_{j=1,j\neq i}^{n}c_{ij}(t)\bigg(\displaystyle\frac{e^{x_{j}(t-\delta_{j}(t))}}{d_{ij}+e^{x_{j}(t-\delta_{j}(t))}}-\frac{e^{y_{j}
(t-\delta_{j}(t))}}{d_{ij}+e^{y_{j}(t-\delta_{j}(t))}}\bigg),i=1,2,\ldots, n.\quad
\end{eqnarray}}Using the mean value theorem, we get
{\setlength\arraycolsep{2pt}\begin{eqnarray}\label{e42}
\left\{
\begin{array}{lll}
e^{x_{i}(t-\tau_{i}(t))}-e^{y_{i}(t-\tau_{i}(t))}=e^{\xi_{i}(t)}(x_{i}(t-\tau_{i}(t))-y_{i}(t-\tau_{i}(t))),\\
\displaystyle\frac{e^{x_{j}(t-\delta_{j}(t))}}{d_{ij}+e^{x_{j}(t-\delta_{j}(t))}}-\frac{e^{y_{j}
(t-\delta_{j}(t))}}{d_{ij}+e^{y_{j}(t-\delta_{j}(t))}}=\frac{e^{\eta_{j}(t)}}
{(d_{ij}+e^{\eta_{j}(t)})^{2}}(x_{j}(t-\delta_{j}(t))-y_{j}(t-\delta_{j}(t))),
\end{array}
\right.
\end{eqnarray}}
where $\xi_{i}(t)$ lies between $x_{i}(t-\tau_{i}(t))$ and $y_{i}(t-\tau_{i}(t))$, $\eta_{j}(t)$ lies between $x_{j}(t-\delta_{j}(t))$ and $y_{j}(t-\delta_{j}(t))$, $i, j=1, 2,\ldots n,\,\, i\neq j$. Then, by use of $\eqref{e42}$, $\eqref{e41}$ can be written as
{\setlength\arraycolsep{2pt}\begin{eqnarray*}
(x_{i}(t)-y_{i}(t))^{\Delta}&=& -b_{i}(t)e^{\xi_{i}(t)}(x_{i}(t-\tau_{i})-y_{i}(t-\tau_{i}))\nonumber\\
&&+\sum_{j=1,j\neq i}^{n}c_{ij}(t)\frac{e^{\eta_{j}(t)}}
{(d_{ij}+e^{\eta_{j}(t)})^{2}}(x_{j}(t-\delta_{j})-y_{j}(t-\delta_{j})),\,\,i=1,2,\ldots, n.
\end{eqnarray*}}
Let $u_{i}(t)=x_{i}(t)-y_{i}(t)$, then
\begin{eqnarray*}
u_{i}^{\Delta}(t)= -b_{i}(t)e^{\xi_{i}(t)}u_{i}(t-\tau_{i}(t))+{\displaystyle\sum_{j=1,j\neq i}^{n}c_{ij}(t)\frac{e^{\eta_{j}(t)}}
{(d_{ij}+e^{\eta_{j}(t)})^{2}}u_{j}(t-\delta_{j}(t))},\,\,i=1, 2,\ldots n.
\end{eqnarray*}

Consider a Lyapunov function
\[V(t)=\sum_{i=1}^{n}V_{i}(t),\]
\[V_{i}(t)=V_{i1}(t)+V_{i2}(t)+V_{i3}(t)+V_{i4}(t)+V_{i5}(t),\]
where
\begin{eqnarray*}
&&V_{i1}(t)=|u_{i}(t)|,\\
&&V_{i2}(t)=\frac{(b_{i}^{u}e^{x_{i}^{M}})^{2}[2\bar{\mu} b_{i}^{u}e^{x_{i}^{M}}+1]}{1-\tau^{\Delta}}\int_{-2\tau^{+}}^{-\tau^{-}}\int_{s+t}^{t}|u_{i}(r)|\Delta r\Delta s,\\
&&V_{i3}(t)={\displaystyle\sum_{j=1,j\neq i}^{n}\frac{c_{ij}^{u}e^{x_{j}^{M}}b_{i}^{u}e^{x_{i}^{M}}[2\bar{\mu} b_{i}^{u}e^{x_{i}^{M}}+1]}
{(d_{ij}+e^{x_{j}^{m}})^{2}(1-\delta^{\Delta})}\int_{-\tau^{+}-\delta^{+}}^{-\delta^{-}}\int_{s+t}^{t}|u_{j}(r)}|\Delta r\Delta s,\\
&&V_{i4}(t)={\displaystyle\sum_{j=1,j\neq i}^{n}\frac{c_{ij}^{u}e^{2x_{j}^{M}}b_{j}^{u}(2\bar{\mu} b_{i}^{u}e^{x_{i}^{M}}+1)}{(d_{ij}+e^{x_{j}^{m}})^{2}(1-\tau^{\Delta})}\int_{-\tau^{+}-\delta^{+}}^{-\tau^{-}}\int_{s+t}^{t}|u_{j}(r)|\Delta r\Delta s},\\
&&V_{i5}(t)={\displaystyle\sum_{j=1,j\neq i}^{n}\frac{(c_{ij}^{u}e^{x_{j}^{M}})^{2}(2\bar{\mu} b_{i}^{u}e^{x_{i}^{M}}+1)}{(d_{ij}+e^{x_{j}^{m}})^{4}(1-\delta^{\Delta})}\int_{-2\delta^{+}}^{-\delta^{-}}\int_{s+t}^{t}|u_{i}(r)|\Delta r\Delta s},
\end{eqnarray*}
then
{\setlength\arraycolsep{2pt}
\begin{eqnarray}\label{e43}
&&D^{+}V_{i1}^{\Delta}(t)\nonumber\\
&\leq&
\mathrm{sign}(u_{i}^{\sigma}(t))u_{i}^{\Delta}(t)\nonumber\\
&=&
\mathrm{sign}(u_{i}^{\sigma}(t))\bigg[-b_{i}(t)e^{\xi_{i}(t)}u_{i}(t-\tau_{i}(t))+{\displaystyle\sum_{j=1,j\neq i}^{n}c_{ij}(t)\frac{e^{\eta_{j}(t)}}{(d_{ij}+e^{\eta_{j}(t)})^{2}}u_{j}(t-\delta_{j}(t))}\bigg]\nonumber\\
&=&
\bigg\{-\mathrm{sign}(u_{i}^{\sigma}(t))b_{i}(t)e^{\xi_{i}(t)}[u_{i}(t)+u_{i}(t-\tau_{i}(t))-u_{i}(t)]\nonumber\\
&&+{\displaystyle\sum_{j=1,j\neq i}^{n}c_{ij}(t)\frac{e^{\eta_{j}(t)}}
{(d_{ij}+e^{\eta_{j}(t)})^{2}}\mathrm{sign}(u_{i}^{\sigma}(t))[u_{j}(t)+u_{j}(t-\delta_{j}(t))-u_{j}(t)]}\bigg\}\nonumber\\
&\leq&
\bigg\{-b_{i}(t)e^{\xi_{i}(t)}\mathrm{sign}(u_{i}^{\sigma}(t))u_{i}(t)+b_{i}^{u}e^{x_{i}^{M}}\int_{t-\tau_{i}(t)}^{t}|u_{i}^{\Delta}(s)|\Delta s+{\displaystyle\sum_{j=1,j\neq i}^{n}c_{ij}^{u}\frac{e^{x_{j}^{M}}}{(d_{ij}+e^{x_{j}^{m}})^{2}}|u_{j}(t)|}\nonumber\\
&&+{\displaystyle\sum_{j=1,j\neq i}^{n}c_{ij}^{u}\frac{e^{x_{j}^{M}}}
{(d_{ij}+e^{x_{j}^{m}})^{2}}\int_{t-\delta_{j}(t)}^{t}|u_{j}^{\Delta}(s)|\Delta s}\bigg\}\nonumber\\
&\leq&
\bigg\{-b_{i}(t)e^{\xi_{i}(t)}\mathrm{sign}(u_{i}^{\sigma}(t))[u_{i}^{\sigma}(t)-\mu(t)u_{i}^{\Delta}(t)]+b_{i}^{u}e^{x_{i}^{M}}
\int_{t-\tau_{i}(t)}^{t}|u_{i}^{\Delta}(s)|\Delta s\nonumber\\
&&+{\displaystyle\sum_{j=1,j\neq i}^{n}c_{ij}^{u}\frac{e^{x_{j}^{M}}}{(d_{ij}+e^{x_{j}^{m}})^{2}}|u_{j}(t)|}+{\displaystyle\sum_{j=1,j\neq i}^{n}c_{ij}^{u}\frac{e^{x_{j}^{M}}}
{(d_{ij}+e^{x_{j}^{m}})^{2}}\int_{t-\delta_{j}(t)}^{t}|u_{j}^{\Delta}(s)|\Delta s}\bigg\}\nonumber\\
&\leq&
\bigg\{-b_{i}^{l}e^{x_{i}^{m}}|u_{i}(t)+\mu(t)u_{i}^{\Delta}(t)|+\bar{\mu} b_{i}^{u}e^{x_{i}^{M}}|u_{i}^{\Delta}(t)|+b_{i}^{u}e^{x_{i}^{M}}\int_{t-\tau_{i}(t)}^{t}|u_{i}^{\Delta}(s)|\Delta s\nonumber\\
&&+{\displaystyle\sum_{j=1,j\neq i}^{n}c_{ij}^{u}\frac{e^{x_{j}^{M}}}{(d_{ij}+e^{x_{j}^{m}})^{2}}|u_{j}(t)|}+{\displaystyle\sum_{j=1,j\neq i}^{n}c_{ij}^{u}\frac{e^{x_{j}^{M}}}
{(d_{ij}+e^{x_{j}^{m}})^{2}}\int_{t-\delta_{j}(t)}^{t}|u_{j}^{\Delta}(s)|\Delta s}\bigg\}\nonumber\\
&\leq&
\bigg\{-b_{i}^{l}e^{x_{i}^{m}}|u_{i}(t)|+2\bar{\mu} b_{i}^{u}e^{x_{i}^{M}}|u_{i}^{\Delta}(t)|+b_{i}^{u}e^{x_{i}^{M}}\int_{t-\tau_{i}(t)}^{t}|u_{i}^{\Delta}(s)|\Delta s\nonumber\\
&&+{\displaystyle\sum_{j=1,j\neq i}^{n}c_{ij}^{u}\frac{e^{x_{j}^{M}}}{(d_{ij}+e^{x_{j}^{m}})^{2}}|u_{j}(t)|}+{\displaystyle\sum_{j=1,j\neq i}^{n}c_{ij}^{u}\frac{e^{x_{j}^{M}}}
{(d_{ij}+e^{x_{j}^{m}})^{2}}\int_{t-\delta_{j}(t)}^{t}|u_{j}^{\Delta}(s)|\Delta s}\bigg\}\nonumber\\
&\leq&
\bigg\{-b_{i}^{l}e^{x_{i}^{m}}|u_{i}(t)|+2\bar{\mu} b_{i}^{u}e^{x_{i}^{M}}\bigg|-b_{i}(t)e^{\xi_{i}(t)}u_{i}(t-\tau_{i}(t))\nonumber\\
&&+{\displaystyle\sum_{j=1,j\neq i}^{n}c_{ij}(t)\frac{e^{\eta_{j}(t)}}{(d_{ij}+e^{\eta_{j}(t)})^{2}}u_{j}(t-\delta_{j}(t))}\bigg|+b_{i}^{u}e^{x_{i}^{M}}\int_{t-\tau_{i}(t)}^{t}|u_{i}^{\Delta}(s)|\Delta s\nonumber\\
&&+{\displaystyle\sum_{j=1,j\neq i}^{n}c_{ij}^{u}\frac{e^{x_{j}^{M}}}{(d_{ij}+e^{x_{j}^{m}})^{2}}|u_{j}(t)|}+{\displaystyle\sum_{j=1,j\neq i}^{n}c_{ij}^{u}\frac{e^{x_{j}^{M}}}
{(d_{ij}+e^{x_{j}^{m}})^{2}}\int_{t-\delta_{j}(t)}^{t}|u_{j}^{\Delta}(s)|\Delta s}\bigg\}\nonumber\\
&\leq&
\bigg\{-b_{i}^{l}e^{x_{i}^{m}}|u_{i}(t)|+2\bar{\mu}(b_{i}^{u}e^{x_{i}^{M}})^{2}|u_{i}(t-\tau_{i}(t))|+{\displaystyle\sum_{j=1,j\neq i}^{n}\frac{2\bar{\mu} b_{i}^{u}c_{ij}^{u}e^{x_{i}^{M}}e^{x_{j}^{M}}}{(d_{ij}+e^{x_{j}^{m}})^{2}}|u_{j}(t-\delta_{j}(t))}|\nonumber\\
&&+b_{i}^{u}e^{x_{i}^{M}}\int_{t-\tau_{i}(t)}^{t}|u_{i}^{\Delta}(s)|\Delta s+{\displaystyle\sum_{j=1,j\neq i}^{n}c_{ij}^{u}\frac{e^{x_{j}^{M}}}{(d_{ij}+e^{x_{j}^{m}})^{2}}|u_{j}(t)|}\nonumber\\
&&+{\displaystyle\sum_{j=1,j\neq i}^{n}c_{ij}^{u}\frac{e^{x_{j}^{M}}}
{(d_{ij}+e^{x_{j}^{m}})^{2}}\int_{t-\delta_{j}(t)}^{t}|u_{j}^{\Delta}(s)|\Delta s}\bigg\}\nonumber\\
&\leq&
\bigg\{-b_{i}^{l}e^{x_{i}^{m}}|u_{i}(t)|+2\bar{\mu}(b_{i}^{u}e^{x_{i}^{M}})^{2}|u_{i}(t)|+2\bar{\mu}(b_{i}^{u}e^{x_{i}^{M}})^{2}\int_{t-\tau_{i}(t)}^{t}
|u_{i}^{\Delta}(s)|\Delta s+b_{i}^{u}e^{x_{i}^{M}}\nonumber\\&&\times\int_{t-\tau_{i}(t)}^{t}|u_{i}^{\Delta}(s)|\Delta s+{\displaystyle\sum_{j=1,j\neq i}^{n}\frac{2\bar{\mu} b_{i}^{u}c_{ij}^{u}e^{x_{i}^{M}}e^{x_{j}^{M}}}{(d_{ij}+e^{x_{j}^{m}})^{2}}|u_{j}(t)|}+{\displaystyle\sum_{j=1,j\neq i}^{n}c_{ij}^{u}\frac{e^{x_{j}^{M}}}{(d_{ij}+e^{x_{j}^{m}})^{2}}|u_{j}(t)|}\nonumber\\
&&+{\displaystyle\sum_{j=1,j\neq i}^{n}\frac{2\bar{\mu} b_{i}^{u}c_{ij}^{u}e^{x_{i}^{M}}e^{x_{j}^{M}}}{(d_{ij}+e^{x_{j}^{m}})^{2}}\int_{t-\delta_{j}(t)}^{t}|u_{j}^{\Delta}(s)|\Delta s}+{\displaystyle\sum_{j=1,j\neq i}^{n}c_{ij}^{u}\frac{e^{x_{j}^{M}}}
{(d_{ij}+e^{x_{j}^{m}})^{2}}\int_{t-\delta_{j}(t)}^{t}|u_{j}^{\Delta}(s)|\Delta s}\bigg\}\nonumber\\
&\leq&
\bigg\{-b_{i}^{l}e^{x_{i}^{m}}|u_{i}(t)|+2\bar{\mu}(b_{i}^{u}e^{x_{i}^{M}})^{2}|u_{i}(t)|+[2\bar{\mu}(b_{i}^{u}e^{x_{i}^{M}})^{2}+b_{i}^{u}e^{x_{i}^{M}}]
\int_{t-\tau_{i}(t)}^{t}|u_{i}^{\Delta}(s)|\Delta s\nonumber\\
&&+{\displaystyle\sum_{j=1,j\neq i}^{n}\frac{c_{ij}^{u}e^{x_{j}^{M}}(2\bar{\mu} b_{i}^{u}e^{x_{i}^{M}}+1)}{(d_{ij}+e^{x_{j}^{m}})^{2}}|u_{j}(t)|}+{\displaystyle\sum_{j=1,j\neq i}^{n}\frac{c_{ij}^{u}e^{x_{j}^{M}}(2\bar{\mu} b_{i}^{u}e^{x_{i}^{M}}+1)}{(d_{ij}+e^{x_{j}^{m}})^{2}}\int_{t-\delta_{j}(t)}^{t}|u_{j}^{\Delta}(s)|\Delta s}\nonumber\\
&\leq&
\bigg\{-b_{i}^{l}e^{x_{i}^{m}}|u_{i}(t)|+2\bar{\mu}(b_{i}^{u}e^{x_{i}^{M}})^{2}|u_{i}(t)|+b_{i}^{u}e^{x_{i}^{M}}[2\bar{\mu} b_{i}^{u}e^{x_{i}^{M}}+1]
\int_{t-\tau_{i}(t)}^{t}\bigg|-b_{i}(s)e^{\xi_{i}(s)}u_{i}(s-\tau_{i}(s))\nonumber\\
&&+{\displaystyle\sum_{j=1,j\neq i}^{n}\frac{c_{ij}(s)e^{\eta_{j}(s)}}
{(d_{ij}+e^{\eta_{j}(s)})^{2}}u_{j}(s-\delta_{j}(s))}\bigg|\Delta s+{\displaystyle\sum_{j=1,j\neq i}^{n}\frac{c_{ij}^{u}e^{x_{j}^{M}}(2\bar{\mu} b_{i}^{u}e^{x_{i}^{M}}+1)}{(d_{ij}+e^{x_{j}^{m}})^{2}}|u_{j}(t)|}\nonumber\\
&&+{\displaystyle\sum_{j=1,j\neq i}^{n}\frac{c_{ij}^{u}e^{x_{j}^{M}}(2\bar{\mu} b_{i}^{u}e^{x_{i}^{M}}+1)}{(d_{ij}+e^{x_{j}^{m}})^{2}}\int_{t-\delta_{j}(t)}^{t}\bigg|-b_{j}(s)e^{\xi_{j}(s)}u_{j}(s-\tau_{j}(t))}\nonumber\\
&&+{\displaystyle \sum_{i=1,j\neq i}^{n}\frac{c_{ji}(s)e^{\eta_{i}(s)}}{(d_{ji}+e^{\eta_{i}(s)})^{2}}u_{i}(s-\delta_{i}(t))}\bigg|\Delta s\nonumber\\
&\leq&
\bigg\{-b_{i}^{l}e^{x_{i}^{m}}|u_{i}(t)|+2\bar{\mu}(b_{i}^{u}e^{x_{i}^{M}})^{2}|u_{i}(t)|+(b_{i}^{u}e^{x_{i}^{M}})^{2}[2\bar{\mu} b_{i}^{u}e^{x_{i}^{M}}+1]
\int_{t-\tau_{i}(t)}^{t}|u_{i}(s-\tau_{i}(s))|\Delta s\nonumber\\
&&+{\displaystyle\sum_{j=1,j\neq i}^{n}\frac{c_{ij}^{u}e^{x_{j}^{M}}b_{i}^{u}e^{x_{i}^{M}}[2\bar{\mu} b_{i}^{u}e^{x_{i}^{M}}+1]}
{(d_{ij}+e^{x_{j}^{m}})^{2}}\int_{t-\tau_{i}(t)}^{t}|u_{j}(s-\delta_{j}(s))}|\Delta s+{\displaystyle\sum_{j=1,j\neq i}^{n}\frac{c_{ij}^{u}e^{x_{j}^{M}}(2\bar{\mu} b_{i}^{u}e^{x_{i}^{M}}+1)}{(d_{ij}+e^{x_{j}^{m}})^{2}}}\nonumber\\
&&\times|u_{j}(t)|+{\displaystyle\sum_{j=1,j\neq i}^{n}\frac{c_{ij}^{u}e^{2x_{j}^{M}}b_{j}^{u}(2\bar{\mu} b_{i}^{u}e^{x_{i}^{M}}+1)}{(d_{ij}+e^{x_{j}^{m}})^{2}}\int_{t-\delta_{j}(t)}^{t}|u_{j}(s-\tau_{j}(t))|\Delta s}\nonumber\\
&&+{\displaystyle \frac{(c_{ij}^{u}e^{x_{j}^{M}})^{2}(2\bar{\mu} b_{i}^{u}e^{x_{i}^{M}}+1)}{(d_{ij}+e^{x_{j}^{m}})^{4}}\int_{t-\delta_{j}(t)}^{t}
|u_{i}(s-\delta_{i}(t))}|\Delta s\nonumber\\
&\leq&
\bigg\{-b_{i}^{l}e^{x_{i}^{m}}|u_{i}(t)|+2\bar{\mu}(b_{i}^{u}e^{x_{i}^{M}})^{2}|u_{i}(t)|+\frac{(b_{i}^{u}e^{x_{i}^{M}})^{2}[2\bar{\mu} b_{i}^{u}e^{x_{i}^{M}}+1]}{1-\tau^{\Delta}}\int_{-2\tau^{+}}^{-\tau^{-}}|u_{i}(s+t)|\Delta s\nonumber\\
&&+{\displaystyle\sum_{j=1,j\neq i}^{n}\frac{c_{ij}^{u}e^{x_{j}^{M}}b_{i}^{u}e^{x_{i}^{M}}[2\bar{\mu} b_{i}^{u}e^{x_{i}^{M}}+1]}
{(d_{ij}+e^{x_{j}^{m}})^{2}(1-\delta^{\Delta})}\int_{-\tau^{+}-\delta^{+}}^{-\delta^{-}}|u_{j}(s+t)}|\Delta s\nonumber\\
&&+{\displaystyle\sum_{j=1,j\neq i}^{n}\frac{c_{ij}^{u}e^{x_{j}^{M}}(2\bar{\mu} b_{i}^{u}e^{x_{i}^{M}}+1)}{(d_{ij}+e^{x_{j}^{m}})^{2}}|u_{j}(t)|}+{\displaystyle\sum_{j=1,j\neq i}^{n}\frac{c_{ij}^{u}e^{2x_{j}^{M}}b_{j}^{u}(2\bar{\mu} b_{i}^{u}e^{x_{i}^{M}}+1)}{(d_{ij}+e^{x_{j}^{m}})^{2}(1-\tau^{\Delta})}\int_{-\tau^{+}-\delta^{+}}^{-\tau^{-}}|u_{j}(s+t)|\Delta s}\nonumber\\
&&+{\displaystyle \sum_{j=1,j\neq i}^{n}\frac{(c_{ij}^{u}e^{x_{j}^{M}})^{2}(2\bar{\mu} b_{i}^{u}e^{x_{i}^{M}}+1)}{(d_{ij}+e^{x_{j}^{m}})^{4}(1-\delta^{\Delta})}\int_{-2\delta^{+}}^{-\delta^{-}}|u_{i}(s+t)|\Delta s},
\end{eqnarray}}
{\setlength\arraycolsep{2pt}
\begin{eqnarray}
D^{+}V_{i2}^{\Delta}(t)
&=&
\frac{(b_{i}^{u}e^{x_{i}^{M}})^{2}[2\bar{\mu} b_{i}^{u}e^{x_{i}^{M}}+1]}{1-\tau^{\Delta}}\int_{-2\tau^{+}}^{-\tau^{-}}[|u_{i}(t)|-|u_{i}(t+s)|]\Delta s\nonumber\\
&=&
\frac{(b_{i}^{u}e^{x_{i}^{M}})^{2}[2\bar{\mu} b_{i}^{u}e^{x_{i}^{M}}+1](2\tau^{+}-\tau^{-})}{1-\tau^{\Delta}}|u_{i}(t)|\nonumber\\
&&-\frac{(b_{i}^{u}e^{x_{i}^{M}})^{2}[2\bar{\mu} b_{i}^{u}e^{x_{i}^{M}}+1]}{1-\tau^{\Delta}}\int_{-2\tau^{+}}^{-\tau^{-}}|u_{i}(t+s)|\Delta s,\label{e44}\\
D^{+}V_{i3}^{\Delta}(t)
&=&
{\displaystyle\sum_{j=1,j\neq i}^{n}\frac{c_{ij}^{u}e^{x_{j}^{M}}b_{i}^{u}e^{x_{i}^{M}}[2\bar{\mu} b_{i}^{u}e^{x_{i}^{M}}+1]}
{(d_{ij}+e^{x_{j}^{m}})^{2}(1-\delta^{\Delta})}\int_{-\tau^{+}-\delta^{+}}^{-\delta^{-}}[|u_{j}(t)|-|u_{j}(t+s)|]\Delta s}\nonumber\\
&=&
{\displaystyle\sum_{j=1,j\neq i}^{n}\frac{c_{ij}^{u}e^{x_{j}^{M}}b_{i}^{u}e^{x_{i}^{M}}[2\bar{\mu} b_{i}^{u}e^{x_{i}^{M}}+1](\tau^{+}+\delta^{+}-\delta^{-})}
{(d_{ij}+e^{x_{j}^{m}})^{2}(1-\delta^{\Delta})}|u_{j}(t)|}\nonumber\\
&&-{\displaystyle\sum_{j=1,j\neq i}^{n}\frac{c_{ij}^{u}e^{x_{j}^{M}}b_{i}^{u}e^{x_{i}^{M}}[2\bar{\mu} b_{i}^{u}e^{x_{i}^{M}}+1]}{(d_{ij}+e^{x_{j}^{m}})^{2}(1-\delta^{\Delta})}\int_{-\tau^{+}-\delta^{+}}^{-\delta^{-}}|u_{j}(t+s)|\Delta s},\label{e45}\\
D^{+}V_{i4}^{\Delta}(t)
&=&
{\displaystyle\sum_{j=1,j\neq i}^{n}\frac{c_{ij}^{u}e^{2x_{j}^{M}}b_{j}^{u}(2\bar{\mu} b_{i}^{u}e^{x_{i}^{M}}+1)}{(d_{ij}+e^{x_{j}^{m}})^{2}(1-\tau^{\Delta})}\int_{-\tau^{+}-\delta^{+}}^{-\tau^{-}}
[|u_{j}(t)|-|u_{j}(t+s)|]\Delta s}\nonumber\\
&=&
{\displaystyle\sum_{j=1,j\neq i}^{n}\frac{c_{ij}^{u}e^{2x_{j}^{M}}b_{j}^{u}(2\bar{\mu} b_{i}^{u}e^{x_{i}^{M}}+1)(\tau^{+}+\delta^{+}-\tau^{-})}{(d_{ij}+e^{x_{j}^{m}})^{2}(1-\tau^{\Delta})}|u_{j}(t)|}\nonumber\\
&&-{\displaystyle\sum_{j=1,j\neq i}^{n}\frac{c_{ij}^{u}e^{2x_{j}^{M}}b_{j}^{u}(2\bar{\mu} b_{i}^{u}e^{x_{i}^{M}}+1)}{(d_{ij}+e^{x_{j}^{m}})^{2}(1-\tau^{\Delta})}|u_{j}(t+s)|\Delta s},\label{e46}\\
D^{+}V_{i5}^{\Delta}(t)
&=&
{\displaystyle\sum_{j=1,j\neq i}^{n}\frac{(c_{ij}^{u}e^{x_{j}^{M}})^{2}(2\bar{\mu} b_{i}^{u}e^{x_{i}^{M}}+1)}{(d_{ij}+e^{x_{j}^{m}})^{4}(1-\delta^{\Delta})}\int_{-2\delta^{+}}^{-\delta^{-}}[|u_{i}(t)|-|u_{i}(t+s)|]\Delta s}\nonumber\\
&=&
{\displaystyle\sum_{j=1,j\neq i}^{n}\frac{(c_{ij}^{u}e^{x_{j}^{M}})^{2}(2\bar{\mu} b_{i}^{u}e^{x_{i}^{M}}+1)(2\delta^{+}-\delta^{-})}{(d_{ij}+e^{x_{j}^{m}})^{4}(1-\delta^{\Delta})}|u_{i}(t)|}\nonumber\\
&&-{\displaystyle\sum_{j=1,j\neq i}^{n}\frac{(c_{ij}^{u}e^{x_{j}^{M}})^{2}(2\bar{\mu} b_{i}^{u}e^{x_{i}^{M}}+1)}{(d_{ij}+e^{x_{j}^{m}})^{4}(1-\delta^{\Delta})}\int_{-\delta_{j}-\delta_{i}}^{-\delta_{i}}|u_{i}(t+s)|\Delta s}\bigg\}.\label{e47}
\end{eqnarray}}
In view of $\eqref{e43}-\eqref{e47}$, we can obtain
{\setlength\arraycolsep{2pt}\begin{eqnarray}\label{e48}
&&D^{+}V_{i}^{\Delta}(t)\nonumber\\
&=&
D^{+}V_{i1}^{\Delta}(t)+D^{+}V_{i2}^{\Delta}(t)+D^{+}V_{i3}^{\Delta}(t)+D^{+}V_{i4}^{\Delta}(t)+D^{+}V_{i5}^{\Delta}(t)\nonumber\\
&\leq&
\bigg\{-b_{i}^{l}e^{x_{i}^{m}}|u_{i}(t)|+2\bar{\mu}(b_{i}^{u}e^{x_{i}^{M}})^{2}|u_{i}(t)|+{\displaystyle\sum_{j=1,j\neq i}^{n}\frac{c_{ij}^{u}e^{x_{j}^{M}}(2\bar{\mu} b_{i}^{u}e^{x_{i}^{M}}+1)}{(d_{ij}+e^{x_{j}^{m}})^{2}}|u_{j}(t)|}\nonumber\\
&&+\frac{(b_{i}^{u}e^{x_{i}^{M}})^{2}[2\bar{\mu} b_{i}^{u}e^{x_{i}^{M}}+1](2\tau^{+}-\tau^{-})}{1-\tau^{\Delta}}|u_{i}(t)|\nonumber\\
&&+{\displaystyle\sum_{j=1,j\neq i}^{n}\frac{c_{ij}^{u}e^{x_{j}^{M}}b_{i}^{u}e^{x_{i}^{M}}[2\bar{\mu} b_{i}^{u}e^{x_{i}^{M}}+1](\tau^{+}+\delta^{+}-\delta^{-})}
{(d_{ij}+e^{x_{j}^{m}})^{2}(1-\delta^{\Delta})}|u_{j}(t)|}\nonumber\\
&&+{\displaystyle\sum_{j=1,j\neq i}^{n}\frac{c_{ij}^{u}e^{2x_{j}^{M}}b_{j}^{u}(2\bar{\mu} b_{i}^{u}e^{x_{i}^{M}}+1)(\tau^{+}+\delta^{+}-\tau^{-})}{(d_{ij}+e^{x_{j}^{m}})^{2}(1-\tau^{\Delta})}|u_{j}(t)|}\nonumber\\
&&+{\displaystyle\sum_{j=1,j\neq i}^{n}\frac{(c_{ij}^{u}e^{x_{j}^{M}})^{2}(2\bar{\mu} b_{i}^{u}e^{x_{i}^{M}}+1)(2\delta^{+}-\delta^{-})}{(d_{ij}+e^{x_{j}^{m}})^{4}(1-\delta^{\Delta})}|u_{i}(t)|}\bigg\}\nonumber\\
&=&
\bigg\{-\bigg[b_{i}^{l}e^{x_{i}^{m}}-2\bar{\mu}(b_{i}^{u}e^{x_{i}^{M}})^{2}-\frac{(b_{i}^{u}e^{x_{i}^{M}})^{2}[2\bar{\mu} b_{i}^{u}e^{x_{i}^{M}}+1](2\tau^{+}-\tau^{-})}{1-\tau^{\Delta}}\nonumber\\
&&-{\displaystyle\sum_{j=1,j\neq i}^{n}\frac{(c_{ij}^{u}e^{x_{j}^{M}})^{2}(2\bar{\mu} b_{i}^{u}e^{x_{i}^{M}}+1)(2\delta^{+}-\delta^{-})}{(d_{ij}+e^{x_{j}^{m}})^{4}(1-\delta^{\Delta})}}\bigg]|u_{i}(t)|+\bigg[{\displaystyle\sum_{j=1,j\neq i}^{n}\frac{c_{ij}^{u}e^{x_{j}^{M}}(2\bar{\mu} b_{i}^{u}e^{x_{i}^{M}}+1)}{(d_{ij}+e^{x_{j}^{m}})^{2}}}\nonumber\\
&&+{\displaystyle\sum_{j=1,j\neq i}^{n}\frac{c_{ij}^{u}e^{x_{j}^{M}}b_{i}^{u}e^{x_{i}^{M}}[2\bar{\mu} b_{i}^{u}e^{x_{i}^{M}}+1](\tau^{+}+\delta^{+}-\delta^{-})}
{(d_{ij}+e^{x_{j}^{m}})^{2}(1-\delta^{\Delta})}}\nonumber\\
&&+{\displaystyle\sum_{j=1,j\neq i}^{n}\frac{c_{ij}^{u}e^{2x_{j}^{M}}b_{j}^{u}(2\bar{\mu} b_{i}^{u}e^{x_{i}^{M}}+1)(\tau^{+}+\delta^{+}-\tau^{-})}{(d_{ij}+e^{x_{j}^{m}})^{2}(1-\tau^{\Delta})}}\bigg]|u_{j}(t)|\bigg\}\nonumber\\
&=&
-\bigg\{b_{i}^{l}e^{x_{i}^{m}}-2\bar{\mu}(b_{i}^{u}e^{x_{i}^{M}})^{2}-\frac{(b_{i}^{u}e^{x_{i}^{M}})^{2}[2\bar{\mu} b_{i}^{u}e^{x_{i}^{M}}+1](2\tau^{+}-\tau^{-})}{1-\tau^{\Delta}}\nonumber\\
&&-{\displaystyle\sum_{j=1,j\neq i}^{n}\frac{(c_{ij}^{u}e^{x_{j}^{M}})^{2}(2\bar{\mu} b_{i}^{u}e^{x_{i}^{M}}+1)(2\delta^{+}-\delta^{-})}{(d_{ij}+e^{x_{j}^{m}})^{4}(1-\delta^{\Delta})}}-{\displaystyle\sum_{j=1,j\neq i}^{n}\frac{c_{ji}^{u}e^{x_{i}^{M}}(2\bar{\mu} b_{j}^{u}e^{x_{j}^{M}}+1)}{(d_{ji}+e^{x_{i}^{m}})^{2}}}\nonumber\\
&&-{\displaystyle\sum_{j=1,j\neq i}^{n}\frac{c_{ji}^{u}e^{x_{i}^{M}}b_{j}^{u}e^{x_{j}^{M}}[2\bar{\mu} b_{j}^{u}e^{x_{j}^{M}}+1](\tau^{+}+\delta^{+}-\delta^{-})}
{(d_{ji}+e^{x_{i}^{m}})^{2}(1-\delta^{\Delta})}}\nonumber\\
&&-{\displaystyle\sum_{j=1,j\neq i}^{n}\frac{c_{ji}^{u}e^{2x_{i}^{M}}b_{i}^{u}(2\bar{\mu} b_{j}^{u}e^{x_{j}^{M}}+1)(\tau^{+}+\delta^{+}-\tau^{-})}{(d_{ji}+e^{x_{i}^{m}})^{2}(1-\tau^{\Delta})}}\bigg\}|u_{i}(t)|\nonumber\\
&=&
-\bigg\{b_{i}^{l}e^{x_{i}^{m}}-2\bar{\mu}(b_{i}^{u}e^{x_{i}^{M}})^{2}-\frac{(b_{i}^{u}e^{x_{i}^{M}})^{2}[2\bar{\mu} b_{i}^{u}e^{x_{i}^{M}}+1](2\tau^{+}-\tau^{-})}{1-\tau^{\Delta}}\nonumber\\
&&-{\displaystyle\sum_{j=1,j\neq i}^{n}\frac{(c_{ij}^{u}e^{x_{j}^{M}})^{2}(2\bar{\mu} b_{i}^{u}e^{x_{i}^{M}}+1)(2\delta^{+}-\delta^{-})}{(d_{ij}+e^{x_{j}^{m}})^{4}(1-\delta^{\Delta})}}-{\displaystyle\sum_{j=1,j\neq i}^{n}\frac{c_{ji}^{u}e^{x_{i}^{M}}(2\bar{\mu} b_{j}^{u}e^{x_{j}^{M}}+1)}{(d_{ji}+e^{x_{i}^{m}})^{2}}}\nonumber\\
&&\times\bigg[1+\frac{b_{j}^{u}e^{x_{j}^{M}}(\tau^{+}+\delta^{+}-\delta^{-})}
{1-\delta^{\Delta}}+\frac{e^{x_{i}^{M}}b_{i}^{u}(\tau^{+}+\delta^{+}-\tau^{-})}{1-\tau^{\Delta}}\bigg]\bigg\}|u_{i}(t)|\nonumber\\
&=&-\gamma_{i}|u_{i}(t)|.
\end{eqnarray}}
From $\eqref{e48}$, we get
\[D^{+}V^{\Delta}(t)\leq\sum_{i=1}^{n}-\gamma_{i}|u_{i}(t)|,\,\,t\in\mathbb{T},\]
then
\[D^{+}V^{\Delta}(t)\leq0,\,\,t\in\mathbb{T}\]
and hence
\begin{eqnarray}\label{e49}
V(t)<V(t_{0}),\,\,t\geq t_{0},\,\,t_{0}\in\mathbb{T}.
 \end{eqnarray}
By use of $\eqref{e48}$ and $\eqref{e49}$, we have
\begin{eqnarray*}
\int_{t_{0}}^{t}\sum_{i=1}^{n}\gamma_{i}|u_{i}(s)|\Delta s\leq V(t_{0})-V(t),\,\,t\geq t_{0},\,\,t_{0}\in\mathbb{T},\,\,i=1, 2,\ldots n.
 \end{eqnarray*}
Consequently,
\begin{eqnarray*}
\int_{t_{0}}^{\infty}|u_{i}(s)|\Delta s\leq \infty,\,\,t_{0}\in\mathbb{T}
 \end{eqnarray*}
and $u_{i}(t)=x_{i}(t)-y_{i}(t)\rightarrow0$ for $t\rightarrow\infty$, $i=1, 2,\ldots n$. This completes the proof.
\end{proof}

\section{Almost periodic solutions}
\setcounter{equation}{0}
{\setlength\arraycolsep{2pt}}
 \indent

In this section, we investigate the existence and uniqueness of   almost
periodic solutions of system $\eqref{e13}$ by use of the almost periodic functional hull theory on time scales.

Let $\{s_{p}\}\subset \Pi$ be any  sequence such that $s_{p}\rightarrow+\infty$ as $p\rightarrow+\infty$. According to
Lemma 2.9, taking a subsequence if necessary, we have
\begin{eqnarray*}
a_{i}(t+s_{p})\rightarrow a_{i}^{\ast}(t),\,\,b_{i}(t+s_{p})\rightarrow b_{i}^{\ast}(t),\,\,c_{ij}(t+s_{p})\rightarrow c_{ij}^{\ast}(t),\\
\tau_{i}(t+s_{p})\rightarrow \tau_{i}^{\ast}(t),\,\,\delta_{j}(t+s_{p})\rightarrow \delta_{j}^{\ast}(t),\,\,p\rightarrow+\infty,
\end{eqnarray*}
for $t\in\mathbb{T}$, $i, j= 1, 2,\ldots,n$, $i\neq j$. Then, we get the hull equations of system (1.3) as follows:
{\setlength\arraycolsep{2pt}
\begin{eqnarray}\label{e51}
x_{i}^{\Delta}(t)=a_{i}^{\ast}(t)-b_{i}^{\ast}(t)e^{x_{i}(t-\tau_{i}^{\ast}(t))}+\sum_{j=1,j\neq i}^{n}c_{ij}^{\ast}(t)\displaystyle\frac{e^{x_{j}(t-\delta_{j}^{\ast}(t))}}{d_{ij}+e^{x_{j}(t-\delta_{j}^{\ast}(t))}},\,\,i= 1, 2,\ldots,n.
\end{eqnarray}}

By use of the almost periodic theory on time scales and Lemma \ref{lem29}, it is easy to obtain the following lemma.
\begin{lemma}\label{lem51}
If system \eqref{e13} satisfies $(H_{1})$-$(H_{4})$, then the hull equations \eqref{e51} also satisfies $(H_{1})$-$(H_{4})$.
\end{lemma}

\begin{theorem}\label{thm51}
Assume that $(H_{1})$-$(H_{4})$ hold, then system \eqref{e13} exists a unique strictly positive almost periodic solution.
\end{theorem}
\begin{proof}
By Lemma \ref{lem211}, in order to prove the existence of a unique strictly positive almost
periodic solution of system $\eqref{e13}$, we only need to prove that each hull equations of system
$\eqref{e13}$ has a unique strictly positive solution.

Firstly, we prove the existence of a strictly positive solution of hull equations $\eqref{e51}$.
By the almost periodicity of $a_{i}(t), b_{i}(t)$ and $c_{ij}(t), i, j= 1, 2,\ldots,n, i\neq j$, for an arbitrary sequence $\omega=\{\omega_{p}\}\subset\Pi$ with $\omega_{p}\rightarrow+\infty$ as $p\rightarrow+\infty$, we have, for $i, j= 1, 2,\ldots,n$, $i\neq j$,
\begin{eqnarray*}
a_{i}^{\ast}(t+\omega_{p})\rightarrow a_{i}^{\ast}(t),\,\,b_{i}^{\ast}(t+\omega_{p})\rightarrow b_{i}^{\ast}(t),\,\,c_{ij}^{\ast}(t+\omega_{p})\rightarrow c_{ij}^{\ast}(t)\\
\tau_{i}^{\ast}(t+\omega_{p})\rightarrow \tau_{i}^{\ast}(t),\,\,\delta_{j}^{\ast}(t+\omega_{p})\rightarrow \delta_{j}^{\ast}(t),\,\,p\rightarrow+\infty.
\end{eqnarray*}
Suppose that $x(t)=(x_{1}(t), x_{2}(t),\ldots,x_{n}(t))$ is any solution of hull equations $\eqref{e51}$. Let $\epsilon$ be an arbitrary small positive number. Since $(H_{1})-(H_{3})$ hold, by the proof of Lemma 3.1, then there exists a $t_{1}\in\mathbb{T} (t_{1}\geq t_{0})$ such that
\begin{eqnarray*}
x_{i}^{m}-\epsilon\leq x_{i}(t)\leq x_{i}^{M}+\epsilon,\,\,{\rm for}\ t\geq t_{1},\,\, i= 1, 2,\ldots,n.
\end{eqnarray*}
Write $x_{ip}(t)=x_{i}(t+\omega_{p})$ for $t\geq t_{1}$, $p= 1, 2,\ldots$, $i= 1, 2,\ldots,n$.
For any positive integer $q$, it is easy to see that there exist sequences $\{x_{ip}(t):p\geq q\}$ such that the sequences $\{x_{ip} (t)\}$ has subsequences, denoted by $\{x_{ip} (t)\}$ again, converging on any finite interval of $\mathbb{T}$ as $p\rightarrow+\infty$, respectively. Thus we have sequences $\{y_{i}(t)\}$ such that
\begin{eqnarray*}
x_{ip} (t)\rightarrow y_{i}(t),\,\,{\rm for}\ t\in\mathbb{T},\,\,{\rm as}\ p\rightarrow+\infty,\,\, i= 1, 2,\ldots,n.
\end{eqnarray*}
Since
{\setlength\arraycolsep{2pt}
\begin{eqnarray*}
x_{ip}^{\Delta}(t)=a_{i}^{\ast}(t+\omega_{p})-b_{i}^{\ast}(t+\omega_{p})e^{x_{i}(t+\omega_{p}-\tau_{i}(t+\omega_{p}))}+\sum_{j=1,j\neq i}^{n}c_{ij}^{\ast}(t+\omega_{p})\displaystyle\frac{e^{x_{j}(t+\omega_{p}-\delta_{j}(t+\omega_{p}))}}{d_{ij}+e^{x_{j}(t+\omega_{p}-\delta_{j}(t+\omega_{p}))}},
\end{eqnarray*}}
by use of Lemma 3.5 in \cite{24}, we have
{\setlength\arraycolsep{2pt}
\begin{eqnarray*}
y_{i}^{\Delta}(t)=a_{i}^{\ast}(t)-b_{i}^{\ast}(t)e^{y_{i}(t-\tau_{i}^{\ast}(t))}+\sum_{j=1,j\neq i}^{n}c_{ij}^{\ast}(t)\displaystyle\frac{e^{y_{j}(t-\delta_{j}^{\ast}(t))}}{d_{ij}+e^{y_{j}(t-\delta_{j}^{\ast}(t))}},\,\,i=1,2,\ldots,n.
\end{eqnarray*}}
We can easily see that $y(t)=(y_{1}(t), y_{2}(t),\ldots,y_{n}(t))$ is a solution of system $\eqref{e51}$ and $x_{i}^{m}-\epsilon\leq y_{i}(t)\leq x_{i}^{M}+\epsilon$ for $t\in\mathbb{T}$, $i= 1, 2,\ldots,n$. Since $\epsilon$ is an arbitrary small positive number, it follows that $x_{i}^{m}\leq y_{i}(t)\leq x_{i}^{M}$ for $t\in\mathbb{T}$, $i= 1, 2,\ldots,n$, which implies that each of the hull equations $\eqref{e51}$ has at least
one strictly positive solution.

Now, we prove the uniqueness of the strictly positive solution of each of the hull
equations $\eqref{e51}$. Suppose that the hull equations $\eqref{e51}$ have two arbitrary strictly positive solutions $x^{\ast}(t)=(x_{1}^{\ast}(t), x_{2}^{\ast}(t),\ldots,x_{n}^{\ast}(t))$ and $y^{\ast}(t)=(y_{1}^{\ast}(t), y_{2}^{\ast}(t),\ldots,y_{n}^{\ast}(t))$. Let $u_{i}^{\ast}(t)=x_{i}^{\ast}(t)-y_{i}^{\ast}(t)$, $i=1, 2,\ldots,n$.
Consider a Lyapunov function
\[V^{\ast}(t)=\sum_{i=1}^{n}V_{i}^{\ast}(t),\]
where\[V_{i}^{\ast}(t)=V_{i1}^{\ast}(t)+V_{i2}^{\ast}(t)+V_{i3}^{\ast}(t)+V_{i4}^{\ast}(t)+V_{i5}^{\ast}(t),\]
\begin{eqnarray*}
&&V_{i1}^{\ast}(t)=|u_{i}^{\ast}(t)|,\\
&&V_{i2}^{\ast}(t)=\frac{(b_{i}^{u}e^{x_{i}^{M}})^{2}[2\bar{\mu} b_{i}^{u}e^{x_{i}^{M}}+1]}{1-\tau^{\Delta}}\int_{-2\tau^{+}}^{-\tau^{-}}\int_{s+t}^{t}|u_{i}^{\ast}(r)|\Delta r\Delta s,\\
&&V_{i3}^{\ast}(t)={\displaystyle\sum_{j=1,j\neq i}^{n}\frac{c_{ij}^{u}e^{x_{j}^{M}}b_{i}^{u}e^{x_{i}^{M}}[2\bar{\mu} b_{i}^{u}e^{x_{i}^{M}}+1]}
{(d_{ij}+e^{x_{j}^{m}})^{2}(1-\delta^{\Delta})}\int_{-\tau^{+}-\delta^{+}}^{-\delta^{-}}\int_{s+t}^{t}|u_{j}^{\ast}(r)}|\Delta r\Delta s,\\
&&V_{i4}^{\ast}(t)={\displaystyle\sum_{j=1,j\neq i}^{n}\frac{c_{ij}^{u}e^{2x_{j}^{M}}b_{j}^{u}(2\bar{\mu} b_{i}^{u}e^{x_{i}^{M}}+1)}{(d_{ij}+e^{x_{j}^{m}})^{2}(1-\tau^{\Delta})}\int_{-\tau^{+}-\delta^{+}}^{-\tau^{-}}\int_{s+t}^{t}|u_{j}^{\ast}(r)|\Delta r\Delta s},\\
&&V_{i5}^{\ast}(t)={\displaystyle\sum_{j=1,j\neq i}^{n}\frac{(c_{ij}^{u}e^{x_{j}^{M}})^{2}(2\bar{\mu} b_{i}^{u}e^{x_{i}^{M}}+1)}{(d_{ij}+e^{x_{j}^{m}})^{4}(1-\delta^{\Delta})}\int_{-2\delta^{+}}^{-\delta^{-}}\int_{s+t}^{t}|u_{i}^{\ast}(r)|\Delta r\Delta s},
\end{eqnarray*}

Similar to the proof of Theorem 4.1, we have
\begin{eqnarray}\label{e52}
D^{+}(V^{\ast})^{\Delta}(t)\leq-\sum_{i=1}^{n}\gamma_{i}|u_{i}^{\ast}(t)|.
\end{eqnarray}
From $\eqref{e52}$, we get
\begin{eqnarray*}
D^{+}(V^{\ast})^{\Delta}(t)\leq0,\,\,t\in\mathbb{T},
\end{eqnarray*}
and hence
\begin{eqnarray*}
V^{\ast}(t)>V^{\ast}(t_{0}),\,\,t\leq t_{0},\,\,t_{0}\in\mathbb{T}.
 \end{eqnarray*}
Then we have
\begin{eqnarray*}
\int_{t}^{t_{0}}\gamma_{i}|u_{i}^{\ast}(s)|\Delta s\leq V^{\ast}(t_{0})-V^{\ast}(t),\,\,t\leq t_{0},\,\,t_{0}\in\mathbb{T},\,\,i=1, 2,\ldots n.
 \end{eqnarray*}
Consequently,
\begin{eqnarray*}
\int_{-\infty}^{t_{0}}|u_{i}^{\ast}(s)|\Delta s\leq \infty,\,\,t_{0}\in\mathbb{T},
 \end{eqnarray*}
and $u_{i}^{\ast}(t)=x_{i}^{\ast}(t)-y_{i}^{\ast}(t)\rightarrow0$ for $t\rightarrow-\infty$, $i=1, 2,\ldots n$.

For $i=1,2,\ldots,n,$ let
\begin{eqnarray*}
P_{i}&=&1+\frac{(b_{i}^{u}e^{x_{i}^{M}})^{2}[2\bar{\mu} b_{i}^{u}e^{x_{i}^{M}}+1](2\tau^{+})^{2}}{1-\tau^{\Delta}}+{\displaystyle\sum_{j=1,j\neq i}^{n}\frac{c_{ij}^{u}e^{x_{j}^{M}}b_{i}^{u}e^{x_{i}^{M}}[2\bar{\mu} b_{i}^{u}e^{x_{i}^{M}}+1](\tau^{+}+\delta^{+})^{2}}
{(d_{ij}+e^{x_{j}^{m}})^{2}(1-\delta^{\Delta})}}\\
&&+{\displaystyle\sum_{j=1,j\neq i}^{n}\frac{c_{ij}^{u}e^{2x_{j}^{M}}b_{j}^{u}(2\bar{\mu} b_{i}^{u}e^{x_{i}^{M}}+1)(\tau^{+}+\delta^{+})^{2}}{(d_{ij}+e^{x_{j}^{m}})^{2}(1-\tau^{\Delta})}}+{\displaystyle\sum_{j=1,j\neq i}^{n}\frac{(c_{ij}^{u}e^{x_{j}^{M}})^{2}(2\bar{\mu} b_{i}^{u}e^{x_{i}^{M}}+1)(2\delta^{+})^{2}}{(d_{ij}+e^{x_{j}^{m}})^{4}(1-\delta^{\Delta})}}.
\end{eqnarray*}
For arbitrary $\epsilon>0$, there exists a positive integer $K_{1}$ such that
\begin{eqnarray*}
|x_{i}^{\ast}(t)-y_{i}^{\ast}(t)|<\frac{\epsilon}{P_{i}},\,\,\forall t<-K_{1},\,\,i=1, 2,\ldots n.
\end{eqnarray*}
Hence, for $i, j=1, 2,\ldots n$ with $i\neq j$, one has
\begin{eqnarray*}
&&V_{i1}^{\ast}(t)\leq\frac{\epsilon}{P_{i}},\,\,\forall t<-K_{1},\\
&&V_{i2}^{\ast}(t)\leq\frac{(b_{i}^{u}e^{x_{i}^{M}})^{2}[2\bar{\mu} b_{i}^{u}e^{x_{i}^{M}}+1](2\tau^{+})^{2}}{1-\tau^{\Delta}}\frac{\epsilon}{P_{i}},\,\,\forall t<-K_{1},\\
&&V_{i3}^{\ast}(t)\leq{\displaystyle\sum_{j=1,j\neq i}^{n}\frac{c_{ij}^{u}e^{x_{j}^{M}}b_{i}^{u}e^{x_{i}^{M}}[2\bar{\mu} b_{i}^{u}e^{x_{i}^{M}}+1](\tau^{+}+\delta^{+})^{2}}
{(d_{ij}+e^{x_{j}^{m}})^{2}(1-\delta^{\Delta})}\frac{\epsilon}{P_{i}}},\,\,\forall t<-K_{1},\\
&&V_{i4}^{\ast}(t)\leq{\displaystyle\sum_{j=1,j\neq i}^{n}\frac{c_{ij}^{u}e^{2x_{j}^{M}}b_{j}^{u}(2\bar{\mu} b_{i}^{u}e^{x_{i}^{M}}+1)(\tau^{+}+\delta^{+})^{2}}{(d_{ij}+e^{x_{j}^{m}})^{2}(1-\tau^{\Delta})}\frac{\epsilon}{P_{i}}},\,\,\forall t<-K_{1},\\
&&V_{i5}^{\ast}(t)\leq{\displaystyle\sum_{j=1,j\neq i}^{n}\frac{(c_{ij}^{u}e^{x_{j}^{M}})^{2}(2\bar{\mu} b_{i}^{u}e^{x_{i}^{M}}+1)(2\delta^{+})^{2}}{(d_{ij}+e^{x_{j}^{m}})^{4}(1-\delta^{\Delta})}\frac{\epsilon}{P_{i}}},\,\,\forall t<-K_{1},
\end{eqnarray*}
which imply that
\begin{eqnarray*}
V^{\ast}(t)<\epsilon,\,\,\forall t<-K_{1}.
\end{eqnarray*}
So,
\begin{eqnarray*}
\lim_{t\rightarrow-\infty}V^{\ast}(t)=0.
\end{eqnarray*}
Note that $V^{\ast}(t)$ is a nonincreasing nonnegative function on $\mathbb{T}$, and that $V^{\ast}(t)=0$. That is
\begin{eqnarray*}
x_{i}^{\ast}(t)=y_{i}^{\ast}(t),\,\,t\in\mathbb{T},\, i=1,2,\ldots,n.
 \end{eqnarray*}
Therefore, each of the hull equations $\eqref{e51}$ has a unique strictly positive solution.
In view of the previous discussion, any of the hull equations $\eqref{e51}$ has a
unique strictly positive solution. By Lemma \ref{lem211}, system $\eqref{e13}$ has a unique strictly positive
almost periodic solution. The proof is completed.
\end{proof}

\section{An example}

\setcounter{equation}{0}
{\setlength\arraycolsep{2pt}}
 \indent

Consider the following multispecies Lotka-Volterra mutualism system with time delays on almost periodic time scale $\mathbb{T}$:
{\setlength\arraycolsep{2pt}\begin{eqnarray}\label{e61}
x_{i}^{\Delta}(t)=a_{i}(t)-b_{i}(t)e^{x_{i}(t-\tau_{i}(t))}+\sum_{j=1,j\neq i}^{2}c_{ij}(t)\displaystyle\frac{e^{x_{j}(t-\delta_{j}(t))}}{d_{ij}+e^{x_{j}(t-\delta_{j}(t))}},\,\,i=1,2,\,\,t\in{\mathbb{T}}.
\end{eqnarray}}

\begin{example}
When we take $\mathbb{T}=\mathbb{R }$, then $\mu(t)=0$. Let
\[a_{1}(t)=0.7-0.02\sin(\sqrt{2}t),\,\,a_{2}(t)=0.61-0.02\sin(\sqrt{3}t),\]
\[b_{1}(t)=0.58-0.01\cos(\sqrt{2}t),\,\,b_{2}(t)=0.55-0.01\sin(\sqrt{2}t),\]
\[\tau_{1}(t)=0.003-0.001\cos t,\,\,\tau_{2}(t)=0.002+0.001\sin t,\]
 \[\delta_{1}(t)=0.004-0.002\cos t,\,\,\delta_{2}(t)=0.002,\]
\[
(c_{ij}(t))_{2\times 2}= \left(
\begin{array}{cccc}
0.06+0.05\sin(2t) & 0.005+0.005\cos(\sqrt{5}t) \\
0.15+0.02\cos(\sqrt{3}t) & 0.08+0.02\sin(\sqrt{2}t) \\
\end{array}\right),\]
\[
(d_{ij})_{2\times 2}= \left(
\begin{array}{cccc}
1.2 & 1 \\
1 & 1.1 \\
\end{array}\right),\]
 then
\[{a_{1}^u}=0.72,\,\,{a_{1}^l}=0.68,\,\, {a_{2}^u}=0.63,\,\, {a_{2}^l}=0.59,\,\,{b_{1}^u}=0.59,\,\,{b_{1}^l}=0.57,\]
\[{b_{2}^u}=0.56,\,\, {b_{2}^l}=0.54,\,\,{c_{11}^u}=0.11,\,\, {c_{12}^u}=0.01,\,\,{c_{21}^u}=0.17,\,\,{c_{22}^u}=0.1.\]
\[\tau^{+}=0.004,\,\,\tau^{-}=0.001,\,\,\delta^{+}=0.006,\,\,\delta^{-}=0.002,\]
\[\tau^{\Delta}=\max\limits_{1\leq i\leq 2}\sup\limits_{t\in\mathbb{R}}\{\tau_{i}'(t)\}=0.001,\,\,\delta^{\Delta}=\max\limits_{1\leq j\leq 2}\sup\limits_{t\in\mathbb{R}}\{\delta_{j}'(t)\}=0.002.\]

By calculating, we have
\[
x_{1}^{M}=\ln\bigg\{\frac{ a_{1}^{u}+c_{12}^{u}}{b_{1}^{l}}\exp\{( a_{1}^{u}+c_{12}^{u})\tau^{+}\}\bigg\}\approx0.250,\,\,x_{1}^{m}=\ln\bigg\{\frac{a_{1}^{l}}{b_{1}^{u}}e^{-b_{1}^{u}e^{x_{1}^{M}}\tau^{+}}\bigg\}\approx0.139,\]
\[x_{2}^{M}=\ln\bigg\{\frac{ a_{2}^{u}+c_{21}^{u}}{b_{2}^{l}}\exp\{( a_{2}^{u}+c_{21}^{u})\tau^{+}\}\bigg\}\approx0.396,\,\,x_{2}^{m}=\ln\bigg\{\frac{a_{2}^{l}}{b_{2}^{u}}e^{-b_{2}^{u}e^{x_{2}^{M}}\tau^{+}}\bigg\}\approx0.052,\]
then
 \begin{eqnarray*}
\gamma_{1}\approx0.604>0,\,\,\gamma_{2}\approx0.527>0.
\end{eqnarray*}
Thus, $(H_{1})-(H_{4})$ are satisfied. According to Theorems 3.1, Theorem 4.2 and Theorem 5.1, system
(6.1) has an unique almost periodic solution, which is globally attractive.
\end{example}

\begin{example}
When we take $\mathbb{T}=\mathbb{Z }$, then $\mu(t)=1$. Let
\[a_{1}(t)=0.3-0.02\sin(\sqrt{2}t),\,\,a_{2}(t)=0.25-0.02\sin(\sqrt{3}t),\]
\[b_{1}(t)=0.27,\,\,b_{2}(t)=0.22,\,\,\tau_{1}(t)=\frac{2+(-1)^{t}}{1000},\,\,\tau_{2}(t)=0.001,\]
 \[\delta_{1}(t)=\frac{3+(-1)^{t}}{1000},\,\,\delta_{2}(t)=0.002,\]
\[
(c_{ij}(t))_{2\times 2}= \left(
\begin{array}{cccc}
0.03+0.02\sin(2t) & 0.01+0.01\cos(\sqrt{5}t) \\
0.006+0.004\cos(\sqrt{3}t) & 0.08+0.02\sin(\sqrt{2}t) \\
\end{array}\right),\]
\[
(d_{ij})_{2\times 2}= \left(
\begin{array}{cccc}
1.2 & 1 \\
1 & 1.1 \\
\end{array}\right),\]
 then
\[{a_{1}^u}=0.32,\,\,{a_{1}^l}=0.28,\,\, {a_{2}^u}=0.27,\,\, {a_{2}^l}=0.23,\,\,{b_{1}^u}={b_{1}^l}=0.27,\]
\[{b_{2}^u}={b_{2}^l}=0.22,\,\,{c_{11}^u}=0.05,\,\, {c_{12}^u}=0.02,\,\,{c_{21}^u}=0.01,\,\,{c_{22}^u}=0.1,\]
\[\tau^{+}=0.003,\,\,\tau^{-}=0.001,\,\,\delta^{+}=0.004,\,\,\delta^{-}=0.002,\]
\[\tau^{\Delta}=\max\limits_{1\leq i\leq 2}\sup\limits_{t\in\mathbb{Z}}\{\Delta\tau_{i}(t)\}=0.002,\,\,\delta^{\Delta}=\max\limits_{1\leq j\leq 2}\sup\limits_{t\in\mathbb{Z}}\{\Delta\delta_{j}(t)\}=0.002.\]

By calculating, we have
\[
x_{1}^{M}=\ln\bigg\{\frac{ a_{1}^{u}+c_{12}^{u}}{b_{1}^{l}}\exp\{-\tau^{+}\log(1-(a_{1}^{u}+c_{12}^{u}))\}\bigg\}\approx0.232,\]
\[x_{1}^{m}=\ln\bigg\{\frac{a_{1}^{l}}{b_{1}^{u}}\exp\bigg\{\tau^{+}\log(1-b_{1}^{u}e^{x_{1}^{M}})\bigg\}\bigg\}\approx0.035,\]
\[x_{2}^{M}=\ln\bigg\{\frac{ a_{2}^{u}+c_{21}^{u}}{b_{2}^{l}}\exp\{-\tau^{+}\log(1-(a_{2}^{u}+c_{21}^{u}))\}\bigg\}\approx0.242,\]
\[x_{i}^{m}=\ln\bigg\{\frac{a_{2}^{l}}{b_{2}^{u}}\exp\bigg\{\tau^{+}\log(1-b_{2}^{u}e^{x_{2}^{M}})\bigg\}\bigg\}\approx0.044,\]
then
 \begin{eqnarray*}
\gamma_{1}\approx0.041>0,\,\,\gamma_{2}\approx0.059>0.
\end{eqnarray*}
Thus, $(H_{1})$-$(H_{4})$ are satisfied. According to Theorems \ref{thm31}, Theorem \ref{thm41} and Theorem \ref{thm51}, system
\eqref{e61} has an unique almost periodic solution, which is globally attractive.
\end{example}


\begin{thebibliography}{}

\bibitem{1} Y.H. Xia, J.D. Cao, S.S. Cheng, Periodic solutions for a Lotka-Volterra mutualism system with several delays, Appl. Math. Modelling 31 (2007) 1960-1969.

\bibitem{2} Y.K. Li, H.T. Zhang, Existence of periodic solutions for a periodic mutualism model on time scales, J. Math. Anal. Appl. 343 (2008) 818-825.

\bibitem{3} Y.M. Wang, Asymptotic behavior of solutions for a Lotka-Volterra mutualism reaction-diffusion system with time delays, Comput. Math. Appl. 58 (2009) 597-604.

\bibitem{4} C.Y. Wang, S. Wang, F.P. Yang, L.R. Li, Global asymptotic stability of positive equilibrium of three-species Lotka-Volterra mutualism models with diffusion and delay effects, Appl. Math. Modelling 34 (2010) 4278-4288.

\bibitem{5} M. Liu, K. Wang, Analysis of a stochastic autonomous mutualism model, J. Math. Anal. Appl. 402 (2013) 392-403.

\bibitem{6} H. Zhang, Y.Q. Li, B. Jing, W.Z. Zhao, Global stability of almost periodic solution of multispecies mutualism system with time delays and impulsive effects, Appl. Math. Comput. 232 (2014) 1138-1150.

\bibitem{m1}  X.P. Yan, W.T. Li, Bifurcation and global periodic solutions in a delayed facultative mutualism system, Physica D 227(1) (2007) 51-69.
\bibitem{m2}   H. Wu, Y. Xia, M. Lin, Existence of positive periodic solution of mutualism system with several delays, Chaos, Solitons   Fractals 36(2) (2008) 487-493.
\bibitem{m3}F. Chen, J. Yang, L. Chen, X. Xie, On a mutualism model with feedback controls, Appl. Math. Comput. 214(2) (2009) 581-587.

\bibitem{7} Y.K. Li, On a periodic mutualism model, The ANZIAM Journal 42 (2001) 569-580.

\bibitem{8} Y.H. Xia, J.D. Cao, H.Y. Zhang, F.D. Chen, Almost periodic solutions $n$-species competitive system with feedback controls, J. Math. Anal. Appl. 294 (2004) 504-522.

\bibitem{9} Z.K. Huang, F.D. Chen, Almost periodic solution of two species model with feedback regulation and infinite delay, J. Eng. Math. 20 (5) (2004) 33-40.

\bibitem{10} C. He, On almost periodic solutions of Lotka-Volterra almost periodic competition systems, Ann. Differential Equations 9 (1) (1993) 26-36.

\bibitem{11} F.D. Chen, Almost periodic solution of the non-autonomous two-species competitive model with stage structure, Appl. Math. Comput. 181 (2006) 685-693.

\bibitem{12} Q. Wang, B.X. Dai, Almost periodic solution for n-species Lotka-Volterra competitive system with delay and feedback controls, Appl. Math. Comput. 200 (1) (2008) 133-146.

\bibitem{13} X.Z. Meng, L.S. Chen, Periodic solution and almost periodic solution for a nonautonomous Lotka-Volterra dispersal system with infinite delay, J. Math. Anal. Appl. 339 (2008) 125-145.

\bibitem{14} Y. Xie, X. Li, Almost periodic solutions of single population model with hereditary effects, Appl. Math. Comput. 203 (2008) 690-697.

\bibitem{23} T.W. Zhang, Y.K. Li, Y. Ye, Persistence and almost periodic solutions for a discrete fishing model with feedback control, Commun. Nonlinear Sci. Numer. Simulat. 16 (2011) 1564-1573.

\bibitem{15} H. Zhang, Y.Q. Li, B. Jing, Global attractivity and almost periodic solution of a discrete mutualism model with delays, Mathematical Methods in the Applied Sciences, 2013, in press, DOI: 10.1002/mma.3039.


\bibitem{22} L. Wu, F. Chen, Z. Li, Permanence and global attractivity of a discrete Schoener¡¯s competition model with delays, Math. Comput. Modelling 49 (2009) 1607-1617.


\bibitem{17} F.D. Chen, Permanence for the discrete mutualism model with time delay, Math. Comput. Modelling 47 (2008) 431-435.

\bibitem{18} Y.Z. Liao, T.W. Zhang, Almost periodic solutions of a discrete mutualism model with variable delays, Discrete Dyn. Nat. Soc.  2012 (2012), Article ID 742102, 27 pages.

\bibitem{m4}Z. Wang, Y. Li, Almost periodic solutions of a discrete mutualism model with feedback controls, Discrete Dyn. Nat. Soc.  2010 (2010), Article ID 286031, 18 pages.

\bibitem{19} H. Zhang, B. Jing, Y.Q. Li, X.F. Fang, Global analysis of almost periodic solution of a discrete multispecies mutualism system, J. Appl. Math. 2014 (2014), Article ID 107968, 12 pages.


\bibitem{20} M. Bohner, A. Peterson, Dynamic equations on time scales: An Introduction with Applications, Birkh$\ddot{\mathrm{a}}$user, Boston, 2001.

\bibitem{zhang1} H.T. Zhang, F.D. Zhang, Permanence of an N-species cooperation system with time delays and feedback controls on time scales, J. Appl. Math. Comput. 46 (2014) 17-31.

\bibitem{24} Y.K. Li, C. Wang, Uniformly almost periodic functions and almost periodic solutions to dynamic equations on time scales, Abs. Appl. Anal. 2011 (2011), Article ID 341520, 22 pages.

\bibitem{25} C.Y. He, Almost periodic differential equations, Higher Education Publishing House, Beijing, 1992 (in Chinese).


\end{thebibliography}
\end{document}